\documentclass[11pt]{amsart}
\usepackage{amsmath,latexsym,amsfonts,amssymb,amsthm}
\usepackage{geometry}
\usepackage{mathrsfs}
\usepackage{graphicx,color}
\usepackage{tikz-cd}
\usepackage[all,cmtip]{xy}
\usepackage{enumitem}
\usepackage{verbatim}
\usepackage{bbm}
\usepackage{bm}
\usepackage{mathtools}
\usepackage[colorlinks=true, linkcolor=red!80!black, urlcolor=purple, citecolor=blue!70!black]{hyperref}

\numberwithin{equation}{section}
\newtheorem{theorem}{Theorem}[section]
\newtheorem{prop}[theorem]{Proposition}

\newtheorem{thm}[theorem]{Theorem}
\newtheorem{lem}[theorem]{Lemma}

\theoremstyle{definition}

\newtheorem{definition}[theorem]{Definition}
 
 \newtheorem{defn}[theorem]{Definition}

\newtheorem{remark}[theorem]{Remark}

\newtheorem{rem}[theorem]{Remark}

\newcommand{\bC}{\mathbb{C}}

\newcommand{\bQ}{\mathbb{Q}}
\newcommand{\bR}{\mathbb{R}}
\newcommand{\bG}{\mathbb{G}}

\newcommand{\calL}{\mathcal{L}}

\newcommand{\calQ}{\mathcal{Q}}
\newcommand{\calX}{\mathcal{X}}

\newcommand{\Hilb}{\mathrm{Hilb}}

\newcommand{\PGL}{\mathrm{PGL}}

\newcommand{\Hom}{\mathrm{Hom}}

\newcommand{\Proj}{\mathrm{Proj}}

\newcommand{\Pic}{\mathrm{Pic}}
\newcommand{\Aut}{\mathrm{Aut}}

\renewcommand{\arraystretch}{1.2}

\newcommand{\bP}{\mathbb{P}}
\newcommand{\bA}{\mathbb{A}}
\newcommand{\cS}{\mathcal{S}}

\newcommand{\SL}{\mathrm{SL}}

\newcommand{\calO}{\mathcal{O}}
\newcommand{\cO}{\mathcal{O}}

\newcommand{\cP}{\mathcal{P}}

\usepackage{graphicx}
\newcommand{\sslash}{\mathbin{/\mkern-6mu/}}

\newcommand{\bZ}{\mathbb{Z}}

\newcommand{\Chow}{\operatorname{Chow}}

\newcommand{\cX}{\mathcal X}
\newcommand{\cD}{\mathcal D}
\newcommand{\cT}{\mathcal T}
\newcommand{\cU}{\mathcal{U}}

\newcommand{\Supp}{\textrm{Supp}}

\newcommand{\rank}{\mathrm{rank}}

\newcommand{\hvol}{\widehat{\mathrm{vol}}}

\newcommand{\ord}{\mathrm{ord}}
\newcommand{\Val}{\mathrm{Val}}

\newcommand{\vol}{\mathrm{vol}}

\newcommand{\ind}{\mathrm{ind}}
\newcommand{\GIT}{\mathrm{GIT}}
\newcommand{\K}{\mathrm{K}}

\newcommand{\CM}{\mathrm{CM}}
\newcommand{\CH}{\mathrm{CH}}
\newcommand{\klt}{\mathrm{klt}}

\newcommand{\lct}{\mathrm{lct}}

\newcommand{\cM}{\mathcal M}

\newcommand{\cL}{\mathcal L}
\newcommand{\cK}{\mathcal{K}}

\newcommand{\Fut}{\mathrm{Fut}}

\newcommand{\bN}{\mathbb{N}}
\newcommand{\ocX}{\overline{\cX}}
\newcommand{\ocD}{\overline{\cD}}

\newcommand{\calC}{\mathcal C}

\newcommand{\bF}{\mathbb{F}}

\newcommand{\oN}{\overline{N}}

\newcommand{\oMK}{\overline{\mathcal{K}}}

\newcommand{\ocK}{\overline{\mathcal{K}}}
\newcommand{\ocL}{\overline{\mathcal{L}}}

\newcommand{\Hodge}{\mathrm{Hodge}}

\newcommand{\codim}{\textrm{codim}}

\newcommand{\sM}{\mathscr{M}}
\newcommand{\oK}{\overline{K}}
\newcommand{\red}{\mathrm{red}}
\newcommand{\tchi}{\tilde{\chi}}
\newcommand{\bs}{\mathrm{bs}}
\newcommand{\QF}{\mathbb{Q}\mathrm{F}}
\newcommand{\pr}{\mathrm{pr}}
\newcommand{\Cl}{\mathrm{Cl}}
\newcommand{\bfP}{\mathbf{P}}
\newcommand{\fM}{\mathfrak{M}}

\newcommand{\sP}{\mathscr{P}}
\newcommand{\sX}{\mathscr{X}}
\newcommand{\sD}{\mathscr{D}}
\newcommand{\sL}{\mathscr{L}}
\newcommand{\sF}{\mathscr{F}}
\newcommand{\chow}{\mathrm{chow}}
\newcommand{\hsF}{\widehat{\sF}}
\newcommand{\oL}{\overline{L}}
\newcommand{\sC}{\mathscr{C}}

\newcommand{\fp}{\mathfrak{p}}

\begin{document}
\title{K-moduli of curves on a quadric surface and K3 surfaces}
\author[Ascher]{Kenneth Ascher}
\address{Department of Mathematics, Princeton University, Princeton, NJ 08544, USA.}
\email{kascher@princeton.edu}
\author[DeVleming]{Kristin DeVleming}
\address{Department of Mathematics, University of California, San Diego, La Jolla, CA 92093, USA.}
\email{kdevleming@ucsd.edu}
\author[Liu]{Yuchen Liu}
\address{Department of Mathematics, Yale University, New Haven, CT 06511, USA.}
\email{yuchen.liu@yale.edu}
\date{\today}

\begin{abstract}
We show that the K-moduli spaces of log Fano pairs $(\bP^1\times\bP^1, cC)$ where $C$ is a $(4,4)$-curve and their wall crossings coincide with the VGIT quotients of $(2,4)$ complete intersection curves in $\bP^3$. This, together with recent results by Laza-O'Grady, implies that these K-moduli spaces form a natural interpolation between the GIT moduli space of $(4,4)$-curves on $\bP^1\times\bP^1$ and the Baily-Borel compactification of moduli of quartic hyperelliptic K3 surfaces.
\end{abstract}

\maketitle{}
\setcounter{tocdepth}{1}
\tableofcontents

\section{Introduction}
The moduli space $\sF$ of polarized K3 surfaces is often constructed as the arithmetic quotient of a Hermitian symmetric domain, and comes with a natural Baily-Borel compactification $\sF \subset \sF^*$. A long standing problem has been to compare this compactification with other compactifications which carry a more geometric meaning, such as those coming from Geometric Invariant Theory (GIT). In particular, if $\fM$ denotes a GIT compactification, there is often a birational period map $\fp: \fM \dashrightarrow \sF^*$ thanks to the global Torelli theorem for K3 surfaces, and a natural question is whether this map can be resolved in a modular way. 

The case of degree two K3 surfaces was worked out by Shah \cite{Sha80} and Looijenga \cite{Loo86}. In particular, Shah constructs a space $\widehat{\fM}$ as a partial Kirwan desingularization of the GIT quotient $\fM$, which Looijenga shows is simultaneously a small partial resolution of $\sF^*$ (a semi-toric compactification in the language of \cite{Loo2}). In particular, there is one space that interpolates between the GIT and Baily-Borel compacitfications. A far-reaching conjectural generalization is proposed by Laza and O'Grady in \cite{LO16}. When $\sF$ is a Type IV locally symmetric variety associated to a lattice of the form $U^2 \oplus D_{N-2}$ (e.g. hyperelliptic quartic K3 surfaces when $N=18$, quartic K3 surfaces when $N=19$, or double EPW-sextics when $N=20$), they conjecture a systematic way to resolve the period map $\fp$ via a series of birational transformations governed by certain divisors present in $\sF^*$. They confirm their conjectures in the case of hyperelliptic quartic K3 surfaces in \cite{LO} (i.e. when $N=18$); we briefly review some of their results (see Section \ref{sec:LOG-VGIT} for a more detailed discussion). 

Let $C$ be a smooth curve in $\bP^1\times\bP^1$ of bidegree $(4,4)$, and let $\pi: X_C \to \bP^1 \times \bP^1$ be the double cover of the quadric surface branched along $C$. The resulting surface $X_C$ is a smooth hyperelliptic polarized K3 surface of degree four, whose polarization is given by the pullback $\pi^*(\calO_{\bP^1}(1) \boxtimes \calO_{\bP^1}(1))$. The corresponding period domain gives a moduli space $\sF \subset \sF^*$. If $\fM:=|\cO_{\bP^1\times\bP^1}(4,4)|\sslash\Aut(\bP^1\times\bP^1)$ denotes the GIT quotient of $(4,4)$ curves on $\bP^1 \times \bP^1$, then there is a birational period map $\fp: \fM \dashrightarrow \sF^*$. In \cite{LO}, Laza and O'Grady described the birational map $\fp$ as a series of explicit wall crossings. Let $\lambda$ denote the Hodge line bundle on $\sF$, and let $\Delta$ = $H/2$, where $H$ is the Heegner divisor parametrizing periods of K3 surfaces which are double covers of a quadric cone.
In this setting, Laza-O'Grady show that one can interpolate between $\sF^*$ and $\fM$ by considering $\sF(\beta) := \Proj R(\lambda + \beta \Delta)$ and varying $ 0 \leq \beta \leq 1$. One aspect of their proof is a variation of GIT (VGIT) study on the moduli space of $(2,4)$-complete intersection curves in $\bP^3$.  Denoting this space by $\fM(t)$, the authors show that each step $\sF(\beta)$ can be realized as the VGIT moduli space $\fM(t)$ for some specific $t(\beta)$. 

If $c\in (0,\frac{1}{2})$ is a rational number, then $(\bP^1 \times \bP^1, cC)$ is a log Fano pair. Recently, it has become apparent that K-stability provides a natural framework to construct compactifications of moduli spaces of log Fano pairs (see e.g. \cite{ADL} or Section \ref{sec:Kmoduli}). With this in mind, our goal in this paper is to use this theory to construct alternative compactifications of the moduli space of smooth $(4,4)$ curves. The framework to study K-moduli of log Fano pairs was established in \cite{ADL}, where we constructed proper good moduli spaces parametrizing $\bQ$-Gorenstein smoothable K-polystable log Fano pairs $(X, cD)$, where $D$ is a rational multiple of $-K_X$ and $c$ is a rational number. Furthermore, we showed that the moduli spaces undergo  wall crossings as the weight $c$ varies. 

Let $\ocK_{c}$ be the connected component of the moduli stack parametrizing K-semistable log Fano pairs which admit $\bQ$-Gorenstein smoothings to $(\bP^1 \times \bP^1, c C)$, where $C$ is a $(4,4)$ curve. By \cite{ADL}, the moduli stack $\oMK_c$ admits a proper good moduli space $\oK_c$. The goal of this paper is to show that this K-moduli space $\oK_c$, and the wall crossings obtained by varying the weight vector $c$, coincide with the wall crossings given by the VGIT $\fM(t)$ under the correspondence $t=\frac{3c}{2c+2}$. In particular, varying the weight $c$ on the K-moduli space $\oK_c$ interpolates between $\fM$ and $\sF^*$, and gives the intermediate spaces an alternative modular meaning.

\begin{theorem}\label{mthm:thmintro}\leavevmode
Let $\ocK_c$ be the moduli stack parametrizing K-semistable (resp. K-polystable) log Fano pairs $(X,cD)$ admitting $\bQ$-Gorenstein smoothings to $(\bP^1\times\bP^1, cC)$ where $C$ is a smooth $(4,4)$ curve. Let $\sM$ be the GIT quotient stack  of $(4,4)$ curves on $\bP^1 \times \bP^1$. Let $\sM(t)$ be the VGIT quotient stack of $(2,4)$ complete intersection curves in $\bP^3$ of slope $t$ (see Definition \ref{def:VGIT}).
\begin{enumerate}
    \item Let $c \in (0, \frac{1}{8})$ be a rational number. Then there is an isomorphism of Artin stacks $\oMK_c \cong \sM$. In particular, a $(4,4)$-curve $C$ on $\bP^1\times\bP^1$ is GIT (poly/semi)semistable if and only if $(\bP^1\times\bP^1, cC)$ is K-(poly/semi)stabile. 
\item Let $c \in (0, \frac{1}{2})$ be a rational number. Then there is an isomorphism of Artin stacks $\oMK_c\cong \sM(t)$  with $t=\frac{3c}{2c+2}$. Moreover, such isomorphisms commute with the wall crossing morphisms for K-moduli stacks $\oMK_c$ and GIT moduli stacks $\sM(t)$.
\end{enumerate}
Moreover, all isomorphisms descend to the level of good moduli spaces. 
\end{theorem}

We note here that the comparison between K-moduli spaces and (V)GIT moduli spaces in various explicit settings has been studied before, such as \cite{MM93, OSS16, SS17, LX19, Fuj17, GMGS18, ADL} (see also Remark \ref{rem:history}).

Combining Theorem \ref{mthm:thmintro} with the main results in \cite{LO}, we obtain the following isomorphisms between moduli spaces and their natural polarizations. In particular, the wall crossing morphisms between our K-moduli spaces $\oK_c$ form a natural interpolation of the period map $\fp:\fM\dashrightarrow \sF^*$. For an explicit description of K-moduli wall crossings, see Remarks \ref{rem:walls-value} and \ref{rem:walls-detail}.

\begin{thm}\label{mthm:spaceiso}
Let $\oK_c$ be the good moduli space parametrizing K-polystable log Fano pairs $(X,cD)$ admitting $\bQ$-Gorenstein smoothings to $(\bP^1\times\bP^1, cC)$ where $C$ is a smooth $(4,4)$ curve. Let $\fM(t)$ be the VGIT quotient space of $(2,4)$ complete intersection curves in $\bP^3$ of slope $t$ (see Definition \ref{def:VGIT}). Then for any rational number $c\in (0,\frac{1}{2})$, we have
\[
\oK_c\cong \fM(t)\cong \sF(\beta), \quad \textrm{where }t=\frac{3c}{2c+2} \textrm{ and }\beta=\min\left\{1,\frac{1-2c}{6c}\right\}.
\]

Moreover, the CM $\bQ$-line bundle on $\oK_c$, the VGIT polarization on $\fM(t)$, and the Laza-O'Grady polarization on $\sF(\beta)$ (i.e. the push forward of $\lambda+\beta\Delta$ under $\sF\dashrightarrow\sF(\beta)$) are all proportional up to positive factors.
\end{thm}

As a consequence of the above theorems and \cite[Theorem 1.1(iv)]{LO}, we identify the final K-moduli space $\oK_{\frac{1}{2}-\epsilon}$ with Looijenga's semitoric compactification $\hsF$ of $\sF$. In part (1) of the following theorem, we give an alternative proof of \cite[Second part of Theorem 1.1(iv)]{LO} using K-stability. Part (2) suggests that $\sF^*$ can be viewed as a moduli space of log Calabi-Yau pairs as expected in \cite[Conjecture 1.8]{ADL}.

\begin{thm}\label{mthm:slcK3}
Let $0<\epsilon,\epsilon'\ll 1$ be two sufficiently small rational numbers. Then we have isomorphisms  $\oK_{\frac{1}{2}-\epsilon}\cong \fM(\frac{1}{2}-\epsilon')\cong \hsF$. Moreover, we have the following.
\begin{enumerate}
    \item The moduli space $\fM(\frac{1}{2}-\epsilon')$ parametrizes quartic hyperelliptic K3 surfaces  with semi-log canonical singularities.
    \item The Hodge line bundle over $\oK_{\frac{1}{2}-\epsilon}$ is semiample with ample model $\sF^*$.
\end{enumerate}
\end{thm}

Finally, we discuss some partial generalizations of Theorem \ref{mthm:thmintro} to higher degree curves on $\bP^1\times\bP^1$ (see also Remark \ref{rem:OSS}).

\begin{thm}\label{mthm:alldeg}
Let $d\geq 3$ be an integer.
Let $\ocK_{d,c}$ be the moduli stack parametrizing K-semistable log Fano pairs $(X,cD)$ admitting $\bQ$-Gorenstein smoothings to $(\bP^1\times\bP^1, cC)$ where $C$ is a smooth $(d,d)$ curve. Let $\sM_d$ be the GIT quotient stack of $(d,d)$ curves on $\bP^1 \times \bP^1$. Let $\sM_d(t)$ be the VGIT quotient stack of $(2,d)$ complete intersection curves in $\bP^3$ of slope $t\in (0, \frac{2}{d})$ (see Definition \ref{def:VGIT-alldeg}).
\begin{enumerate}
    \item Let $c \in (0, \frac{1}{2d})$ be a rational number. Then there is an isomorphism of Artin stacks $\oMK_{d,c} \cong \sM_d$. In particular, $C$ is GIT (poly/semi)semistable on $\bP^1\times\bP^1$ if and only if $(\bP^1\times\bP^1, cC)$ is K-(poly/semi)stabile. 
\item Let $c \in (0, \frac{4-\sqrt{2}}{2d})$ be a rational number. Then there is an isomorphism of Artin stacks $\oMK_{d,c}\cong \sM_d(t)$  with $t=\frac{6c}{dc+4}$. Moreover, such isomorphisms commute with the wall crossing morphisms for K-moduli stacks $\oMK_{d,c}$ and GIT moduli stacks $\sM_d(t)$.
\end{enumerate}
\end{thm}

\subsection*{Organization} For the remainder of this paper, $c$ (and thus $t$ and $\beta$) will always denote a rational number. 
This paper is organized as follows. In Section \ref{sec:prelim} we recall the definitions of K-stability, normalized volumes, and the CM-line bundle. We also recall the main results of \cite{ADL}, and define the relevant moduli functor. In Section \ref{sec:LO}, we recall the background on K3 surfaces and review the main results of \cite{LO}. In Section \ref{sec:surfaces}, we determine which surfaces can appear as degenerations of $\bP^1 \times \bP^1$ on the boundary of the K-moduli spaces. Key ingredients are Theorems \ref{thm:indexbound} and \ref{thm:surfaces} which bound the Gorenstein indices of singular surfaces using normalized volumes. In Section \ref{sec:main}, we compare the GIT compactification with the K-stability compactification, and study the wall crossings that appear for K-moduli. In particular, we present the proofs of Theorems \ref{mthm:thmintro}, \ref{mthm:spaceiso}, and \ref{mthm:slcK3}. These are achieved by the index estimates mentioned above, computation of CM line bundles, and a modification of Paul-Tian's criterion \cite{PT06} (see also \cite{OSS16, ADL}) to work over non-proper bases (see also \cite{SS17}). 
Note that the VGIT of $(2,4)$-complete intersections in $\bP^3$ for a general slope does not provide a $\bQ$-Gorenstein flat log Fano family over a proper base, but only such a family over the complete intersection locus as a quasi-projective variety. This creates an issue that the usual Paul-Tian's criterion cannot be directly applied. In order to resolve this issue, we trace the change of K/VGIT stability conditions along their wall crossings, and argue that their polystable replacements indeed coincide. Finally, in Section \ref{sec:generaldegree}, we discuss some generalizations for higher degree curves on $\bP^1 \times \bP^1$ and prove Theorem \ref{mthm:alldeg}.

\subsection*{Acknowledgements}
We would like to thank David Jensen, Radu Laza, Zhiyuan Li, Xiaowei Wang, and Chenyang Xu for helpful discussions. We also thank the referee for many valuable suggestions. This material is based upon work supported
by the National Science Foundation under Grant No. DMS-1440140 while the authors were in residence at the Mathematical Sciences Research Institute in Berkeley, California,
during the Spring 2019 semester. The authors were supported in part by the American Insitute of Mathematics as part of the AIM SQuaREs program.  KA was partially supported by an NSF Postdoctoral Fellowship and NSF Grant No. DMS-2001408. KD was partially supported by the Gamelin Endowed Postdoctoral Fellowship of the MSRI.  YL was partially supported by the Della Pietra Endowed Postdoctoral Fellowship of the MSRI and the NSF Grant No. DMS-2001317.

\section{Preliminaries}\label{sec:prelim}

Throughout this paper, we work over the field of complex numbers $\bC$, and all schemes are assumed to be of finite type over $\bC$. A variety is a separated integral scheme of finite type over $\bC$.

\subsection{K-stability of log Fano pairs}
We first recall necessary background to define K-stability of log Fano pairs. 

\begin{defn}
Let $X$ be a normal variety and let $D$ be an effective $\bQ$-divisor on $X$.
We say such $(X,D)$ is a \emph{log pair}. 
If $X$ is projective and $-(K_X+D)$ is $\bQ$-Cartier ample, then the log pair $(X,D)$ is called a \emph{log Fano pair}. The variety $X$ is a \emph{$\bQ$-Fano variety} if $(X,0)$ is a klt log Fano pair.
\end{defn}

Next, we recall the definition of K-stability of log Fano pairs.

\begin{defn}[\cite{Tia97, Don02, Li15, LX14, OS15}]
Let $(X,D)$ be a log Fano pair. Let $L$ be an ample line bundle on $X$ such that $L\sim_{\bQ}-l(K_X+D)$ for some $l\in \bQ_{>0}$.
\begin{enumerate}[label=(\alph*)]
\item A \emph{normal test configuration} $(\cX,\cD;\cL)/\bA^1$ of $(X,D;L)$ consists of the following data:
\begin{itemize}
 \item a normal variety $\cX$ together with a flat projective morphism $\pi:\cX\to \bA^1$;
 \item a $\pi$-ample line bundle $\cL$ on $\cX$;
 \item a $\bG_m$-action on $(\cX;\cL)$ such that $\pi$ is $\bG_m$-equivariant with respect to the standard action of $\bG_m$ on $\bA^1$ via multiplication;
 \item $(\cX\setminus\cX_0;\cL|_{\cX\setminus\cX_0})$
 is $\bG_m$-equivariantly isomorphic to $(X;L)\times(\bA^1\setminus\{0\})$.
\item an effective $\bQ$-divisor $\cD$ on $\cX$ such that $\cD$ is the Zariski closure of $D\times(\bA^1\setminus\{0\})$ under the identification between $\cX\setminus\cX_0$ and $X\times(\bA^1\setminus\{0\})$.
\end{itemize}

A normal test configuration is called a \emph{product} test configuration if \[
(\cX,\cD;\cL)\cong(X\times\bA^1,D\times\bA^1;\mathrm{pr}_1^* L\otimes\cO_{\cX}(k\cX_0))
\] for some $k\in\bZ$. A product test configuration is called a \emph{trivial} test configuration if the above isomorphism is $\bG_m$-equivariant with respect to the trivial $\bG_m$-action on $X$ and the standard $\bG_m$-action on $\bA^1$ via multiplication.

\item For a normal test configuration $(\cX,\cD;\cL)/\bA^1$ of $(X,D)$, denote its natural compactification over $\bP^1$ by $(\ocX,\ocD;\ocL)$. 
The \emph{generalized Futaki invariant} of $(\cX,\cD;\cL)/\bA^1$ is defined by the following intersection formula due to \cite{Wan12, Oda13a}:
   \[
    \Fut(\cX,\cD;\cL):=\frac{1}{(-(K_X+D))^n}\left(\frac{n}{n+1}\cdot\frac{(\bar{\cL}^{n+1})}{l^{n+1}}+\frac{(\bar{\cL}^n\cdot (K_{\bar{\cX}/\bP^1}+\bar{\cD}))}{l^n}\right).
   \]

\item The log Fano pair $(X,D)$ is said to be:
\begin{enumerate}[label=(\roman*)]
    \item \emph{K-semistable} if $\Fut(\cX,\cD;\cL)\geq 0$ for any normal test configuration $(\cX,\cD;\cL)/\bA^1$ and any $l\in\bQ_{>0}$ such that $L$ is Cartier; 
    
    \item  \emph{K-stable} if it is K-semistable and $\Fut(\cX,\cD;\cL)=0$ for a normal test configuration $(\cX,\cD;\cL)/\bA^1$ if and only if it is a trivial test configuration; and
 
\item \emph{K-polystable} if it is K-semistable and $\Fut(\cX,\cD;\cL)=0$ for a normal test configuration $(\cX,\cD;\cL)/\bA^1$ if and only if it is a product test configuration.
\end{enumerate}
 
\item 
Let $(X,D)$ be a klt log Fano pair.  Then a normal test configuration $(\cX,\cD;\cL)/\bA^1$ is called a \emph{special test configuration} if $\cL\sim_{\bQ}-l(K_{\cX/\bA^1}+\cD)$ and $(\cX,\cD+\cX_0)$ is plt. In this case, we say that $(X,D)$ \emph{specially degenerates to} $(\cX_0,\cD_0)$ which is necessarily a klt log Fano pair.
\end{enumerate}
\end{defn}


\begin{rem}\leavevmode
\begin{enumerate}
    \item The concept of K-(semi/poly)stability of log Fano pairs can also be defined via test configurations that are possibly non-normal. For the general definitions we refer to \cite[Section 2.1]{ADL}. By \cite[Proposition 3.15]{BHJ17}, we know that generalized Futaki invariants will not increase under normalization of test configurations.
    \item  Odaka proved in \cite{Oda12} that any K-semistable log Fano pair is klt. By the work of Li and Xu \cite{LX14}, to test K-(poly/semi)stability of a klt log Fano pair, it suffices to test the sign of generalized Futaki invariants only on special test configurations.
\end{enumerate}

\end{rem}

The following lemma is very useful in the proof of Theorem \ref{mthm:thmintro}.

\begin{lem}\label{lem:zerofut}\leavevmode
\begin{enumerate}
  \item\cite{Kem78} Let $G$ be a reductive group acting 
  on a polarized projective scheme $(Y,L)$.
  Let $y\in Y$ be a closed point. Let $\sigma:\bG_m\to G$
  be a 1-PS. Denote by $y'=\lim_{t\to 0}\sigma(t)\cdot y$.
  If $y$ is GIT semistable and $\mu^{L}(y,\sigma)=0$,
  then $y'$ is also GIT semistable.
  \item\cite[Lemma 3.1]{LWX18}
  Let $(X,D)$ be a log Fano pair. Let $(\cX,\cD;\cL)/\bA^1$
  be a normal test configuration of $(X,D)$.
  If $(X,D)$ is K-semistable and $\Fut(\cX,\cD;\cL)=0$,
  then $(\cX,\cD;\cL)/\bA^1$ is a special test configuration and $(\cX_0,\cD_0)$ is also K-semistable.
 \end{enumerate}

\end{lem}

\subsection{Normalized volumes}
 In this section, we consider a klt singularity $x\in (X,D)$, that is, a klt log pair $(X,D)$ with a closed point $x\in X$.  Recall that a \emph{valuation $v$ on $X$ centered at $x$} is a real valuation of $\bC(X)$ such that the valuation ring $\cO_v$ dominates $\cO_{X,x}$ as local rings. The set of such valuations is denoted by $\Val_{X,x}$.

We briefly review normalized volume of valuations as introduced by Chi Li \cite{Li18}. See \cite{LLX18} for a survey on recent developments.

\begin{defn}
Let $x\in (X,D)$ be an $n$-dimensional klt singularity.
\begin{enumerate}[label=(\alph*)]
    \item The \emph{volume} is a function $\vol_{X,x}:\Val_{X,x}\to \bR_{\geq 0}$ defined in \cite{ELS03} as 
\[
\vol_{X,x}(v):=\lim_{k\to\infty}\frac{\dim_{\bC}\cO_{X,x}/\{f\in\cO_{X,x}\mid v(f)\geq k\}}{k^n/n!}.
\]
\item The \emph{log discrepancy}  is a function $A_{(X,D)}:\Val_{X,x}\to \bR_{>0}\cup\{+\infty\}$ defined in \cite{JM12, BdFFU15}. If $v=a\cdot\ord_E$ where $a\in\bR_{>0}$ and $E$ is a prime divisor over $X$ centered at $x$, then 
\[
A_{(X,D)}(v)=a(1+\ord_E(K_{Y}-\pi^*(K_X+D))),
\]
where $\pi:Y\to X$ provides a birational model $Y$ of $X$ containing $E$ as a divisor. In this paper, we only deal with divisorial valuations.
\item The \emph{normalized volume} is a function $\hvol_{(X,D),x}:\Val_{X,x}\to \bR_{>0}\cup\{+\infty\}$ defined in \cite{Li18} as
\[
\hvol_{(X,D),x}(v):=\begin{cases}
A_{(X,D)}(v)^n\cdot\vol_{X,x}(v) & \textrm{ if } A_{(X,D)}(v)<+\infty\\
+\infty & \textrm{ if }A_{(X,D)}(v)=+\infty
\end{cases}
\]

The \emph{local volume} of a klt singularity $x\in (X,D)$ is defined as 
\[
\hvol(x,X,D):=\min_{v\in\Val_{X,x}}\hvol_{(X,D),x}(v).
\]
Note that the existence of a normalized volume minimizer is proven in \cite{Blu18}. From \cite{LX16} we know that  $\hvol(x,X,D)$ can be approximated by normalized volume of divisorial valuations.
\end{enumerate}

\end{defn}

The following theorem from \cite{LL16} generalizing \cite[Theorem 1.1]{Fuj15} and \cite[Theorem 1.2]{Liu18} is crucial. Note that it also follows from the valuative criterion for K-semistability by Fujita \cite{Fuj16} and C. Li \cite{Li17}.

\begin{thm}[{\cite[Proposition 4.6]{LL16}}]\label{thm:local-vol-global}
Let $(X,D)$ be a K-semistable log Fano pair of dimension $n$. Then for any closed point $x\in X$, we have 
\[
(-K_X-D)^n\leq \left(1+\frac{1}{n}\right)^n\hvol(x,X,D).
\]
\end{thm}

\subsection{CM line bundles}
The Chow-Mumford (CM) line bundle of a flat family of polarized projective varieties was introduced algebraically by Tian \cite{Tia97} as a functorial line bundle over the base.
We start with the definition of CM line bundles due to Paul and Tian \cite{PT06, PT09} using the Knudsen-Mumford expansion (see also \cite{FS90, FR06}). 
In order to define CM line bundles for families of log Fano pairs over reduced bases, we need to use the concept of relative Mumford divisors from \cite[Definition 1]{Kol19} (see also \cite{Kol18}). 

\begin{defn}[relative Mumford divisors]
Let $f: \cX\to T$ be a morphism between schemes. Assume that $f$ has $S_2$ fibers of pure dimension $n$. A closed subscheme $\cD$ of $\cX$ is a \emph{relative Mumford divisor} over $T$ if there is an open subset $\cU\subset \cX$ such that 
\begin{enumerate}
    \item $\codim_{\cX_t} (\cX_t\setminus \cU_t)\geq 2$ for any $t\in T$;
    \item $\cD|_{\cU}$ is a relative Cartier divisor;
    \item $\cD$ is the scheme theoretic closure of $\cD|_{\cU}$;
    \item $\cX_t$ is smooth at generic points of $\Supp(\cD_t)$ for any $t\in T$.
\end{enumerate}
A \emph{relative Mumford $\bQ$-divisor} is a formal $\bQ$-linear combination of relative Mumford divisors.
\end{defn}

\begin{defn}[log CM line bundle]\label{defn:logCM}
Let $f:\cX\to T$ be a proper flat morphism of connected schemes. Assume that $f$ has $S_2$ fibers of pure dimension $n$. Let $\cL$ be an $f$-ample line bundle on $\cX$. 

A result of Knudsen-Mumford \cite{KM76} says that there exists line bundles $\lambda_j=\lambda_{j}(\cX,\cL)$ on $T$ such that for all $k$,
\[
\det f_!(\cL^k)=\lambda_{n+1}^{\binom{k}{n+1}}\otimes\lambda_n^{\binom{k}{n}}\otimes\cdots\otimes\lambda_0.
\]
By flatness, the Hilbert polynomial $\chi(\cX_t,\cL_t^k)=a_0 k^n+a_1 k^{n-1}+ O(k^{n-2})$. 
Then the \emph{CM line bundle} of the data $(f:\cX\to T,\cL)$ is defined as
\[
\lambda_{\CM,f,\cL}:=\lambda_{n+1}^{\mu+n(n+1)}\otimes\lambda_n^{-2(n+1)},
\]
where $\mu=\mu(\cX,\cL):=\frac{2a_1}{a_0}$.
The \emph{Chow line bundle} is defined as
\[
\lambda_{\Chow,f,\cL}:=\lambda_{n+1}.
\]

Let $\cD:=\sum_{i=1}^{k} c_i\cD_i$ be a relative Mumford $\bQ$-divisor on $\cX$ over $T$ where each $\cD_i$ is a relative Mumford divisor and $c_i\in[0,1]\cap\bQ$. We also assume that each $\cD_i$ is flat over $T$. 

The \emph{log CM $\bQ$-line bundle} of the data $(f:\cX\to T, \cL,\cD:=\sum_{i=1}^k c_i\cD_i)$ is defined as 
\[
 \lambda_{\CM,f,\cD,\cL}:=\lambda_{\CM,f,\cL}-\frac{n(\cL_t^{n-1}\cdot\cD_t)}{(\cL_t^n)}\lambda_{\Chow,f,\cL}+(n+1)\lambda_{\Chow,f|_{\cD},\cL|_{\cD}},
\]
where \[
(\cL_t^{n-1}\cdot\cD_t):=\sum_{i=1}^k c_i (\cL_t^{n-1}\cdot\cD_{i,t}),\quad
  \lambda_{\Chow,f|_{\cD},\cL|_{\cD}}:=\bigotimes_{i=1}^k\lambda_{\Chow,f|_{\cD_i},\cL|_{\cD_i}}^{\otimes c_i}.
  \]
\end{defn}

Note that if $T$ is not connected, then we define the log CM line bundle on each connected component of $T$ as in Definition \ref{defn:logCM}.



Next, we recall the concept of $\bQ$-Gorenstein flat families of log Fano pairs over reduced base schemes.

\begin{defn}\label{defn:qgorfamily}
Let $T$ be a reduced scheme. Let $f:\cX\to T$ be a proper flat morphism with normal, geometrically connected fibers of pure dimension $n$. Let $\cD$ be an effective relative Mumford $\bQ$-divisor on $\cX$ over $T$. We say that $f:(\cX,\cD)\to T$ is a \emph{$\bQ$-Gorenstein flat family of log Fano pairs} if 
$-(K_{\cX/T}+\cD)$ is $\bQ$-Cartier and $f$-ample.

\end{defn}

We consider the following class of log Fano pairs as objects of our moduli problems.

\begin{definition}\label{defn:qgorsmoothable}
 Let $c,r$ be positive rational numbers such that $c<\min\{1, r^{-1}\}$. 
 A log Fano pair $(X,cD)$ is \emph{$\bQ$-Gorenstein smoothable} if there exists a $\bQ$-Gorenstein flat family of log Fano pairs $\pi:(\cX,c\cD)\to B$ over a pointed smooth curve $(0\in B)$
 such that the following holds:
 \begin{itemize}
  \item $\cD$ is a relative Mumford divisor over $B$;
  \item Both $-K_{\cX/B}$ and $\cD$ are $\bQ$-Cartier, $\pi$-ample and  $\cD\sim_{\bQ,\pi}-rK_{\cX/B}$;
  \item Both $\pi$ and $\pi|_{\cD}$ are smooth morphisms
  over $B\setminus\{0\}$;
  \item $(\cX_0,c\cD_0)\cong (X,cD)$, and $X$ has klt singularities.
 \end{itemize} 
 
A $\bQ$-Gorenstein flat family of log Fano pairs $f:(\cX,c\cD)\to T$ over a reduced scheme $T$ is called a \emph{$\bQ$-Gorenstein smoothable log Fano family} if $\cD$ is a $\bQ$-Cartier relative Mumford divisor over $T$, and 
all fibers of $f$ are $\bQ$-Gorenstein smoothable log Fano pairs.
\end{definition}

\begin{lem}\label{lem:flatness}
For $c,r\in\bQ_{>0}$ with $cr<1$, let $(\cX,c\cD)\to B$ be a $\bQ$-Gorenstein smoothable log Fano family over a smooth curve $B$ where $\cD\sim_{\bQ,B}-rK_{\cX/B}$. Then the function $B\ni b\mapsto h^0(\cX_b,\cO_{\cX_b}(m\cD_b))$ is constant for any $m\in \bZ$.
\end{lem}

\begin{proof}
By inversion of adjunction we know that $\cX$ has klt singularities. Since $\cD$ and $\cD_b$ are $\bQ$-Cartier Weil divisors on $\cX$ and $\cX_b$ respectively, we know that both $\cO_{\cX}(m\cD)$ and $\cO_{\cX_b}(m\cD_b)$ are Cohen-Macaulay for any $m\in \bZ$ by \cite[Corollary 5.25]{KM98}. Hence $\cO_{\cX_b}(m\cD_b)\cong \cO_{\cX}(m\cD)\otimes \cO_{\cX_b}$ for any $m\in \bZ$. By Kawamata-Viehweg vanishing, we know that $H^i(\cX_b,\cO_{\cX_b}(m\cD_b))=0$ for any $b\in B$ and $m,i\geq 1$. Hence the statement for $m>0$ follows from the semi-continuity theorem and flatness of $\cO_{\cX}(m\cD)$ over $B$, while it is obvious for $m\leq 0$.
\end{proof}

\begin{prop}\label{prop:smoothable-flat}
Let $f:(\cX,c\cD)\to T$ be a $\bQ$-Gorenstein smoothable log Fano family over a reduced scheme $T$. Then $\cD$ is flat over $T$. 
\end{prop}

\begin{proof}
For simplicity, we assume that $T$ is connected.
By \cite[Theorem 4.33]{Kol17} there exists a locally closed decomposition $T'\to T$, such that for any morphism $q:W\to T$, the divisorial pull-back $f_W: (\cX_W, \cD_W)\to W$ of $f:(\cX,\cD)\to T$ satisfies that $\cD_W$ is flat over $W$ if and only if $q$ factors as $q: W\to T'\to T$. It is clear that $\cD$ is flat over $T$ if and only if $T'=T$. Thus it suffices to show that for any morphism $B\to T$ from a smooth curve $B$, the divisorial pull-back $f_B:(\cX_B,\cD_B)\to B$ of $f$ satisfies that $\cD_B$ is flat over $B$. It is clear that $f_B$ is also a $\bQ$-Gorenstein smoothable log Fano family. By the proof of Lemma \ref{lem:flatness}, we know that $\cO_{\cX_B}(-\cD_B)$ is flat over $B$, and its fiber over $b\in B$ is isomorphic to $\cO_{\cX_{B,b}}(-\cD_{B,b})$ which is $S_2$. Hence \cite[Definition-Lemma 4.19]{Kol17} implies that $\cD_B\to B$ is flat. This finishes the proof.
\end{proof}

\begin{defn}
We define the \emph{CM $\bQ$-line bundle} of a $\bQ$-Gorenstein smoothable log Fano family $f:(\cX,c\cD)\to T$ over a reduced scheme $T$ to be  $\lambda_{\CM,f,c\cD}:=l^{-n}\lambda_{\CM,f,c\cD,\cL}$, where $\cL:=-l(K_{\cX/T}+c\cD)$ is an $f$-ample Cartier divisor on $\cX$ for some $l\in\bZ_{>0}$. Here $\cD$ is flat over $T$ by Proposition \ref{prop:smoothable-flat}.
\end{defn}

Note that the CM $\bQ$-line bundle of a $\bQ$-Gorenstein smoothable log Fano family is functorial under reduced base change by the functoriality of Chow line bundles and $\lambda_n$ according to \cite{KM76, CP18}. 

We can now recall the definition of the Hodge line bundle for a smoothable log Calabi-Yau fibration of Fano type. From the definition it is clear that the Hodge line bundle is functorial under reduced base change.

\begin{defn}\label{defn:hodge}
For $c,r\in \bQ_{>0}$ with $cr<1$, let $f:(\cX,c\cD)\to T$ be a $\bQ$-Gorenstein smoothable log Fano family over a reduced scheme $T$ where  $\cD\sim_{\bQ,f} -rK_{\cX/T}$. The Hodge $\bQ$-line bundle $\lambda_{\Hodge,f,r^{-1}\cD}$ is defined as the $\bQ$-linear equivalence class of $\bQ$-Cartier $\bQ$-divisors on $T$ such that
\[
K_{\cX/T}+r^{-1}\cD\sim_{\bQ}f^*\lambda_{\Hodge,f,r^{-1}\cD}.
\]
\end{defn}

The following proposition relates CM $\bQ$-line bundles and the Hodge $\bQ$-line bundle.

\begin{prop}\label{prop:cm-interpolation}\cite[Proposition 2.25]{ADL}
With the notation of Definition \ref{defn:hodge}, for any rational number $0\leq c<r^{-1}$ we have 
\begin{equation}\label{eq:cm-hodge}
(1-cr)^{-n}\lambda_{\CM,f,c\cD}=(1-cr)\lambda_{\CM,f}+ cr(n+1)(-K_{\cX_t})^n\lambda_{\Hodge,f,r^{-1}\cD}.
\end{equation}
\end{prop}

The next criterion is important when checking K-stability in explicit families. It is a partial generalization of \cite[Theorem 1]{PT06} and \cite[Theorem 3.4]{OSS16}.

\begin{thm}\label{thm:paultian}\cite[Theorem 2.22]{ADL}
Let $f:(\cX,c\cD)\to T$ be a $\bQ$-Gorenstein smoothable log Fano family over a normal projective variety $T$. Let $G$ be a reductive group acting on $\cX$ and $T$ such that $\cD$ is $G$-invariant and $f$ is $G$-equivariant. 
Assume in addition that 
\begin{enumerate}[label=(\alph*)]
\item if $\Aut(\cX_t,\cD_t)$ is finite for $t\in T$ then the stabilizer subgroup $G_t$ is also finite;
\item if $(\cX_t,\cD_t)\cong (\cX_{t'}, \cD_{t'})$ for $t,t'\in T$, then $t'\in G\cdot t$;
\item $\lambda_{\CM,f,c\cD}$ is an ample $\bQ$-line bundle on $T$.
\end{enumerate}
Then $t\in T$ is GIT (poly/semi)stable with respect to the $G$-linearized $\bQ$-line bundle $\lambda_{\CM,f,c\cD}$ if $(\cX_t, c\cD_t)$ is a K-(poly/semi)stable log Fano pair.
\end{thm}

The following proposition provides an intersection formula for log CM line bundles. For the case without divisors this was proven by Paul and Tian \cite{PT06}. The current statement follows from {\cite[Proposition 3.7]{CP18}}.

\begin{prop}\label{prop:logCM2}\cite[Proposition 2.23]{ADL}
Let $f:(\cX,c\cD)\to T$ be a $\bQ$-Gorenstein smoothable log Fano family of relative dimension $n$ over a normal proper variety $T$. Then
\begin{equation}\label{eq:CM-intersection}
 \mathrm{c}_1(\lambda_{\CM,f,c\cD})=-f_*((-K_{\cX/T}-c\cD)^{n+1}).
\end{equation}
\end{prop}

\subsection{K-moduli spaces of log Fano pairs}\label{sec:Kmoduli}
    In this subsection, we gather recent results on the construction of K-moduli spaces of log Fano pairs.
    
    In \cite{ADL}, we construct K-moduli stacks (resp. proper good moduli spaces) of $\bQ$-Gorenstein smoothable K-semistable (resp. K-polystable) log Fano pairs $(X,cD)$ where $D \sim_{\bQ} -rK_X$, and $c$ is a rational number. 

 \begin{thm}\label{thm:lwxlog}\cite[Theorem 3.1 and Remark 3.25]{ADL}
 Let $\chi_0$ be the Hilbert polynomial of an anti-canonically polarized Fano manifold. Fix $r\in\bQ_{>0}$ and a rational number $c\in (0,\min\{1,r^{-1}\})$. Consider the following moduli pseudo-functor over reduced base $S$:
\[
\cK\cM_{\chi_0,r,c}(S)=\left\{(\cX,\cD)/S\left| \begin{array}{l}(\cX,c\cD)/S\textrm{ is a $\bQ$-Gorenstein smoothable log Fano family,}\\ \cD\sim_{S,\bQ}-rK_{\cX/S},~\textrm{each fiber $(\cX_s,c\cD_s)$ is K-semistable,}\\ \textrm{and $\chi(\cX_s,\cO_{\cX_s}(-kK_{\cX_s}))=\chi_0(k)$ for $k$ sufficiently divisible.}\end{array}\right.\right\}.
\]
Then there exists a reduced Artin stack  $\cK\cM_{\chi_0,r,c}$ (called a \emph{K-moduli stack}) of finite type over $\bC$ representing the above moduli pseudo-functor. In particular, the $\bC$-points of $\cK\cM_{\chi_0,r,c}$ parametrize K-semistable $\bQ$-Gorenstein smoothable log Fano pairs $(X,cD)$ with Hilbert polynomial $\chi(X,\cO_X(-mK_X))=\chi_0(m)$ for sufficiently divisible $m$ and $D\sim_{\bQ}-rK_X$.

Moreover, the Artin stack $\cK\cM_{\chi_0,r,c}$ admits a good moduli space $KM_{\chi_0,r,c}$ (called a \emph{K-moduli space}) as a proper reduced scheme of finite type over $\bC$, whose closed points parametrize K-polystable log Fano pairs.
\end{thm}

By \cite[Proposition 3.35]{ADL}, we know that the universal log Fano family over $\cK\cM_{\chi_0,r,c}$ provides a CM $\bQ$-line bundle $\lambda_c$ and a Hodge $\bQ$-line bundle $\lambda_{c,\Hodge}$ over $\cK\cM_{\chi_0,r,c}$ which descend to $\bQ$-line bundles $\Lambda_c$ and $\Lambda_{c,\Hodge}$ over the good moduli space $KM_{\chi_0,r,c}$. 
Recently, it was shown by Xu and Zhuang that the above K-moduli spaces are projective with ample CM $\bQ$-line bundles.

\begin{thm}\label{thm:projectivity}\cite[Theorem 7.10]{XZ19}
The CM $\bQ$-line bundle $\Lambda_c$ over $KM_{\chi_0,r,c}$ is ample. Hence $KM_{\chi_0,r,c}$ is a projective scheme.
\end{thm}

\begin{rem}\label{rem:history}
The explicit study of K-moduli originated in \cite{MM93} with the case of degree four del Pezzo surfaces, and other degree del Pezzo surfaces were later studied in \cite{OSS16}. Since then, this area has seen rapid growth (see e.g. \cite{SS17, LX19, Fuj17, GMGS18, ADL, Liu20}). In all of the aforementioned cases, the \emph{smoothable} condition was necessary. 

If we drop the $\bQ$-Gorenstein smoothable condition, then K-moduli stacks and spaces of log Fano pairs with fixed numerical conditions (such as volume and finite coefficient set) exist as Artin stacks and projective  schemes, respectively. For a precise statement, see e.g. \cite[Theorem 2.21]{XZ19} and \cite[Theorem 1.3]{LXZ21}. These follow from recent works of \cite{Jia17, CP18, BX18, ABHLX19, BLX19, Xu19, XZ19, XZ20, BHLLX20, LXZ21}. Since the work \cite{LXZ21} establishing the properness of K-moduli spaces appeared after the first version of this article was posted on the arXiv, we restrict our discussion to the smoothable components of K-moduli spaces with reduced structure.
\end{rem}

The following result shows that any K-moduli stack $\cK\cM_{\chi_0,r,c}$  parametrizing two-dimensional $\bQ$-Gorenstein smoothable log Fano pairs is always smooth. For the special case of plane curves on $\bP^2$, see \cite[Proposition 4.6]{ADL}.

\begin{thm}\label{thm:modnormal}
Let $\chi_0$ be the Hilbert polynomial of an anti-canonically polarized smooth del Pezzo surface. Fix $r\in\bQ_{>0}$ and a rational number $c\in (0,\min\{1,r^{-1}\})$. Then the K-moduli stack $\cK\cM_{\chi_0,r,c}$ is isomorphic to the quotient stack of a smooth scheme by a projective general linear group. In particular,  $\cK\cM_{\chi_0,r,c}$ is smooth and $KM_{\chi_0,r,c}$ is normal.
\end{thm}

\begin{proof}
Fix a sufficiently divisible $m\in\bZ_{>0}$. Denote by 
\[
\chi(k):=\chi_0(mk),\quad \tchi(k)=\chi_0(mk)-\chi_0(mk-r), \quad \textrm{and}\quad N_m:=\chi_0(m)-1.
\]
Recall that in \cite[Section 3.1]{ADL}, we construct a locally closed subscheme $Z^{\klt}$ of the relative Hilbert scheme $\Hilb_{\chi}(\bP^{N_m})\times\Hilb_{\tchi}(\bP^{N_m})$ which parametrizes $\bQ$-Gorenstein smoothable log Fano pairs $(X,cD)$ such that they are embedded into $\bP^{N_m}$ by $|-mK_X|$ and $X$ is klt. 
Denote by $Z$ the dense open subscheme of $Z^{\klt}$ parametrizing $(X,D)$ where both $X$ and $D$ are smooth.
Let $Z_c^\circ$ be the Zariski open subset of $Z^{\klt}$ parametrizing K-semistable log Fano pairs $(X,cD)$. Denote by $Z_c^{\red}$ the reduced scheme supported on $Z_c^\circ$. Then $\cK\cM_{\chi_0,r,c}$ is defined as the quotient stack $[Z_c^{\red}/\PGL(N_m+1)]$. Hence it suffices to show that $Z^{\klt}$ is smooth which would then imply that $Z_c^{\red}$ is smooth. The argument below is inspired by \cite[Lemma 9.7]{ADL}.

Denote by $Z^{\QF}$ the locally closed subscheme of $\Hilb_{\chi}(\bP^{N_m})$ parametrizing $\bQ$-Gorenstein smoothable $\bQ$-Fano varieties $X$ that are embedded into $\bP^{N_m}$ by $|-mK_X|$. Since we are in dimension $2$, any point $\Hilb(X)\in Z^{\QF}$ corresponds to a log del Pezzo surface $X$ with only $T$-singularities.  Hence $X$ has unobstructed $\bQ$-Gorenstein deformations by \cite[Theorem 8.2]{Hac04}, \cite[Proposition 3.1]{hp}, and \cite[Lemma 6]{ACC+}. Thus $Z^{\QF}$ is a smooth scheme. Denote by $Z^{\rm sm}$ the Zariski open subset of $Z^{\QF}$ parametrizing smooth Fano manifolds $X$ such that there exists a smooth divisor $D\sim_{\bQ}-rK_X$. The openness of $Z^{\rm sm}$ follows from openness of smoothness, $H^0(X,\cO_X(D))$ being constant since $H^i(X,\cO_X(D))=0$ for $i \geq 1$ by Kodaira vanishing, and the fact that smooth families of Fano manifolds have locally constant Picard groups. Denote by $Z^{\bs}:=\overline{Z^{\rm sm}}\cap Z^{\QF}$. Hence $Z^{\bs}$ is the disjoint union of some connected components of $Z^{\QF}$.
Denote the first projection by $\pr_1: Z^{\klt}\to \Hilb_{\chi}(\bP^{N_m})$. Clearly $\pr_1(Z^{\klt})$ is contained in $Z^{\QF}$. We claim that $\pr_1(Z^{\klt})=Z^{\bs}$, and that the restriction morphism $\pr_1:Z^{\klt}\to Z^{\bs}$ is proper and smooth.

We first show that $\pr_1(Z^{\klt})=Z^{\bs}$ and $\pr_1:Z^{\klt}\to Z^{\bs}$ is proper. Since $Z$ is a dense open subset of $Z^{\klt}$, we know that
\[
Z^{\rm sm}=\pr_1(Z)\subset\pr_1(Z^{\klt})\subset \overline{\pr_1(Z)}\cap Z^{\QF}=\overline{Z^{\rm sm}}\cap Z^{\QF}=Z^{\bs}.
\]
Hence the surjectivity of $\pr_1:Z^{\klt}\to Z^{\bs}$ would follow from its properness. We will verify properness by checking the existence part of valuative criterion. Let $0\in B$ be a pointed curve with $B^\circ:=B\setminus\{0\}$. Consider two morphisms $f^\circ: B^\circ\to Z^{\klt}$ and $g:B\to Z^{\QF}$ such that $g|_{B^\circ}=\pr_1\circ f^\circ$. It suffices to show that $f^\circ$ extends to $f:B\to Z^{\klt}$ such that $g=\pr_1\circ f$. We have a $\bQ$-Gorenstein smoothable family $p:\cX\to B$ induced by $g$, and a $\bQ$-Cartier Weil divisor $\cD^\circ$ on $\cX^\circ:=p^{-1}(B^\circ)$ induced by $f^\circ$ whose support does not contain any fiber $\cX_b$, and $\cD^\circ\sim_{\bQ,B^\circ}-rK_{\cX^\circ/B^\circ}$. We define $\cD:=\overline{\cD^\circ}$. Then, by taking Zariski closure, it is clear that $\cD\sim_{\bQ, B}-rK_{\cX/B}$ since $\cX_0$ is a $\pi$-linearly trivial Cartier prime divisor on $\cX$. Thus $(\cX,\cD)\to B$ is a $\bQ$-Gorenstein smoothable log Fano family. This finishes proving properness and surjectivity of $\pr_1:Z^{\klt}\to Z^{\bs}$.

Finally, we will show that $\pr_1:Z^{\klt}\to Z^{\bs}$ is a smooth morphism. Indeed, we will show that it is a smooth $\bP^{N_r}$-fibration where $N_r:=\chi_0(r)-1$. If  $\Hilb(X,D)\in \pr_1^{-1}(Z^{\rm sm})$, then we know that $h^0(X,\cO_X(D))=\chi(X,\cO_X(D))=\chi_0(r)$ since $H^i(X,\cO_X(D))=0$ for any $i\geq 1$ by Kodaira vanishing.
Hence the fiber over $\Hilb(X)\in Z^{\rm sm}$ is isomorphic to $\bP(H^0(X,\cO_X(D)))\cong\bP^{N_r}$. Hence we may restrict to the case when $\Hilb(X)\in Z^{\bs}\setminus Z^{\rm sm}$. Assume that $(X,D)\in Z^{\klt}$ is $\bQ$-Gorenstein smoothable where $\pi: (\cX,\cD)\to B$ is a $\bQ$-Gorenstein smoothing over a pointed curve $0\in B$ with $(\cX_0,\cD_0)\cong (X,D)$. Then by Lemma \ref{lem:flatness}  we know that $\pi_*\cO_{\cX}(\cD)$ is locally free with fiber over $b\in B$ isomorphic to $H^0(\cX_b,\cO_{\cX_b}(\cD_b))$. Hence it is easy to conclude that for any effective Weil divisor $D'\sim D$ the pair $(X,D')$ is also $\bQ$-Gorenstein smoothable. Since the Weil divisor class group $\Cl(X)$ of $X$ is finitely generated, we know that there are only finitely many Weil divisor classes $[D]$ such that $[D]= -r[K_X]$ in $\Cl(X)\otimes_{\bZ} \bQ$. Hence the fiber $\pr_1^{-1}(\Hilb(X))$ is isomorphic to a disjoint union of finitely many copies of $\bP^{N_r}$. However, since $\pr_1:Z^{\klt}\to Z^{\bs}$ is proper with connected fibers over a dense open subset $Z^{\rm sm}$ and $Z^{\bs}$ is normal, taking Stein factorization yields that $\pr_1$ has connected fibers everywhere. Hence $\pr_1^{-1}(\Hilb(X))\cong\bP^{N_r}$ for any $\Hilb(X)\in Z^{\bs}$. Therefore, $\pr_1$ has smooth fibers and smooth base which implies that $Z^{\klt}$ is Cohen-Macaulay. Hence, miracle flatness implies that $\pr_1$ is flat and hence smooth. The proof is finished.
\end{proof}



\section{Overview of previous results, Laza-O'Grady, and VGIT}\label{sec:LO} We refer the reader to \cite{LO16, LO18b, LO} for more details.

\subsection{Hyperelliptic K3 surfaces of degree 4}
A \emph{K3 surface} $X$ is a connected projective surface with Du Val singularities such that $\omega_X\cong\cO_X$ and $H^1(X,\cO_X)=0$. A K3 surface $X$ together with an ample line bundle $L$ on $X$ is called a \emph{polarized K3 surface} $(X,L)$ of degree $(L^2)$. 
A polarized K3 surface $(X,L)$ is \emph{hyperelliptic} if the map $\varphi_L: X \dashrightarrow |L|^\vee$ is regular, and is a double cover of its image. All hyperelliptic quartic K3 surfaces are obtained by the following procedure (see \cite[Remark 2.1.3]{LO16}). Consider a normal quadric surface $Q \subset \bP^3$, and $B \in |\omega_Q^{-2}|$ with ADE singularities (in particular, GIT stable when $Q\cong\bP^1\times\bP^1$). Then the double cover $\pi: X \to Q$ ramified over $B$ is a hyperelliptic quartic K3 with polarization $L = \pi^*\calO_{Q}(1)$ and at worst ADE singularities. 

Given a smooth $(4,4)$ curve $C$ on $\bP^1 \times \bP^1$, the double cover $\pi: X_C \to \bP^1 \times \bP^1$ ramified over $C$ is a hyperelliptic polarized K3 surface of degree 4. The polarization is given by $L_C = \pi^*\calO_{\bP^1\times\bP^1}(1,1)$. One can ask how the GIT moduli space of $(4,4)$ curves on $\bP^1 \times \bP^1$ compares to the moduli space of hyperelliptic K3 surfaces of degree 4 constructed via periods. 

\subsection{Moduli of K3 surfaces}\label{sec:mk3}
Let $\Lambda$ be the lattice $U^2 \oplus D_{16}$, where $U$ is the hyperbolic plane and $D_{16}$ is the negative definite lattice corresponding to Dynkin diagram $D_{16}$. Let $\mathscr{D} = \{ |\sigma| \in \bP(\Lambda \otimes \bC) \mid \sigma^2 = 0, (\sigma + \overline{\sigma})^2 > 0\}.$ The connected component $\mathscr{D}^+$ is a type IV bounded symmetric domain. Let $\Gamma(\Lambda) = O^+(\Lambda) < O(\Lambda)$ be the index two subgroup mapping $\mathscr{D}^+$ to itself. We define the locally symmetric variety $\sF = \Gamma \setminus \mathscr{D}^+$, and we let $\sF \subset \sF^*$ be its Baily-Borel compactification (see \cite[Section 3.1]{LZ}). 

It turns out that $\sF$ can be identified as the period space for hyperelliptic quartic K3 surfaces (see \cite[Remark 2.2.4]{LO16}). The rough idea is that $\sF$ sits inside a larger period domain $\sF^\prime$ which serves as a moduli space for quartic K3 surfaces,  and $\sF$ is naturally isomorphic to a divisor in $\sF^\prime$ whose points correspond to the periods of the hyperelliptic K3 surfaces. 

Let $\fM$ denote the GIT moduli space of $(4,4)$ curves on $\bP^1 \times \bP^1$.  Shah proved that $(4,4)$ curves with ADE singularities are GIT-stable and, by associating to $C$ the corresponding period point of the K3 surface, one obtains a rational period map $\fp: \fM \dashrightarrow \sF^*$ (\cite[Theorem 4.8]{Sha80}). By the Global Torelli theorem, the period map $\fp$ is actually birational. Laza-O'Grady show that the indeterminacy locus of $\fp$ is a subset of $\fM$ of dimension 7 (see e.g. \cite[Corollary 4.10]{LO}). The goal of Laza-O'Grady's work is to describe this birational map explicitly, as a series of flips and divisorial contractions.

The intersection of $\sF$ and the image of the regular locus of $\fp$ is $\sF \setminus H_h$, where $H_h$ is a Heegner divisor. Geometrically, it parametrizes periods of hyperelliptic K3 surfaces which are double covers of a quadric cone, and is defined as follows. The vector $w \in \Lambda$ is hyperbolic if $w^2 = -4$ and the divisibility $div(w) = 2$ (the positive generator of $(w, \Lambda)$). The Heegner divisor $H_h \subset \sF$ is the locus of $O^+(\Lambda)$-equivalence classes of points $[\sigma] \in \mathscr{D}^+$ such that $\sigma^\perp$ contains a hyperbolic vector. 

\subsection{Results of Laza-O'Grady and VGIT for (2,4)-complete intersections in $\bP^3$}\label{sec:LOG-VGIT}
As mentioned in the introduction, Laza and O'Grady propose a conjectural framework to interpolate between GIT and Baily-Borel compactifications of moduli spaces which are Type IV locally symmetric domains $\sF(N)$ associated to lattices of the form $U^2 \oplus D_{N_2}$ (see \cite{LO16}). These include, for example, K3 surface of degree four ($N = 19$) and EPW sextics ($N = 20$). 

Let $\lambda(N)$ denote the Hodge line bundle on $\sF(N)$ and let $\Delta(N)$ denote a geometrically meaningful (e.g. Noether-Lefschetz) divisor. By the work of Baily-Borel, the compact space $\sF^*(N)$ can always be identified with $\Proj R(\sF(N), \lambda(N))$. Moreover, Looijenga showed that $\Proj R(\sF(N), \lambda(N) + \Delta(N))$ can often be identified with $\fM$. The main prediction of Laza and O'Grady is that the ring of sections $R(\sF(N), \lambda(N) + \beta \Delta(N))$ is finitely generated for $\beta \in [0, 1] \cap \bQ$. Moreover, they give a prediction for the ``walls'' where the moduli spaces change, thus predicting a natural interpolation between $\sF^*$ and $\fM$.

From now on, we restrict to the case where $N=18$, that is, the case of hyperelliptic quartic K3 surfaces.  In this case, if $\Delta = H_h /2$ (introduced in Section \ref{sec:mk3}), then it was shown in \cite{LO16} that $\fM \cong \Proj R(\sF, \lambda + \Delta)$. Let $\beta \in [0,1] \cap \bQ$. In \cite{LO}, Laza-O'Grady prove that the ring of sections $R(\sF, \lambda + \beta \Delta)$ is finitely generated, and therefore $\sF(\beta) = \Proj R(\sF, \lambda + \beta \Delta)$ can be viewed as a projective variety interpolating between the GIT and Baily-Borel moduli spaces. Moreover, they calculate the set of critical values, and  show that the birational period map is the composition of explicitly understood divisorial contractions and flips.  In fact, they show that the intermediate spaces arise from variation of GIT (VGIT). They also show that the first step in their program produces $\hsF=\sF(\epsilon)\to \sF^*$ as the $\bQ$-Cartierization of $H_h \subset \sF^*$ for $0<\epsilon\ll 1$. In particular, this gives a small partial resolution $\hsF$ of $\sF^*$ which parametrizes hyperelliptic quartic K3 surfaces with slc singularities. In what follows, we review VGIT and their results in further detail.

We now introduce the VGIT $\sM(t)$, largely modeled off of \cite[Section 5]{LO}. A smooth $(2,4)$-complete intersection inside $\bP^3$ determines $X_C$, a smooth hyperelliptic K3  of degree 4. Let $U$ be the parameter space for all $(2,4)$-complete intersection closed subschemes in $\bP^3$. Then $U$ has a natural action of $\SL(4)$, though we note that $U$ is not projective. We let $E$ be the vector bundle over $|\calO_{\bP^3}(2)|$ whose fiber over $Q \in  |\calO_{\bP^3}(2)|$ is given by $H^0(Q, \calO_Q(4))$. Then $U \subseteq \bP(E)$ and $\codim_{\bP(E)}\bP(E) \setminus U \geq 2$.

 There is a map $\chow: U \to \Chow$ to the Chow variety parametrizing 1-dimensional cycles inside $\bP^3$. We denote by $\Chow_{(2,4)}$ the closure of the image of $\chow$. Note then that there is  regular embedding: 
\[ U \hookrightarrow \bP(E) \times \Chow_{(2,4)}. \]

Next, we describe the the universal family of log Fano pairs over $U$. We need this to set up the VGIT and in Section \ref{sec:CM} to compute the CM line bundle.  We begin by considering the following diagram 
\begin{center}
\begin{tikzcd}
(\sX, \sD) \arrow[d, "f"] \arrow[r, hook] &  \bP^3 \times \bP(E) \arrow[dl, "p_2"] \\
\bP(E) \arrow[d, "\pi"] &  \\
\bP(H^0(\bP^3, \calO(2))) = \bP^9 & \\
\end{tikzcd}
\end{center}

We let $p_1$ (resp. $p_2$) denote the first (resp. second) projections, and let $f: (\sX, \sD) \to \bP(E)$ be the universal family over $\bP(E)$, where we view $(\sX, \sD) \subseteq \bP^3 \times \bP(E)$. We let  $\calQ \subset \bP^3 \times \bP^9$ denote the universal family over $\bP^9$ with morphism $\phi: \calQ \to \bP^9$, and let $E = \phi_* \calO_{\calQ}(4,0)$. Pointwise, we have

\begin{center}
\begin{tikzcd}
H^0(Q, \calO_Q(4)) \arrow[r,hook] \arrow[d, mapsto] & \bP(E) \arrow[d] \\
{[Q]} \arrow[r,hook] & \bP^9\\
\end{tikzcd}
\end{center}
Using the notation of Laza-O'Grady (see \cite[(5.2)]{LO}), we denote by $\eta: = \pi^*\cO_{\bP^9}(1)$ and $\xi := \cO_{\bP(E)}(1)$. We recall the following result of Benoist. 

\begin{prop}\cite[Theorem 2.7]{benoist}\label{prop:benoist2,4}
If $t \in \bQ$, then the $\bQ$-Cartier class $\eta + t\xi$ on $\bP(E)$ is ample if and only if $t \in (0, \frac{1}{3}) \cap \bQ$. 
\end{prop}

We now set up the VGIT, following \cite[Section 5.1]{LO}. Let $\sP$ denote the closure of $U$ in $\bP(E) \times \Chow_{(2,4)}$.  Let $p_1$ and $p_2$ be the first and second projections from $\sP$ to $\bP(E)$ and $\Chow_{(2,4)}$, respectively.   The action of $\SL(4)$ on $\bP^3$ extends to an action on $\mathscr{P}$. To construct a GIT quotient, we thus need to specify a $\SL(4)$ linearized ample line bundle on $\sP$. 

Fix a rational number $0 < \delta < \frac{1}{6}$.  For $t \in (\delta, 1/2) \cap \bQ$, consider the $\mathbb{Q}$-line bundle 
\[N_t := \frac{1 - 2t}{1-2\delta} p_1^*(\eta + \delta \xi) + \frac{t - \delta}{2(1-2\delta)} p_2^*L_{\infty} ,\]
where $L_{\infty}$ is the restriction of the natural polarization of the Chow variety to $\Chow_{(2,4)}$.  One can check that $N_t$ is ample for $\delta < t < \frac{1}{2}$ and semiample for $t = \frac{1}{2}$. 

\begin{definition}\label{def:VGIT}
Let $\delta \in \mathbb{Q}$ satisfy $0 < \delta < \frac{1}{6}$. For each $t \in (\delta, \frac{1}{2}] \cap \bQ$, we define the VGIT quotient stack $\sM(t)$ of slope $t$ to be, and the VGIT quotient space $\fM(t)$ of slope $t$ to be
\[ \sM(t) := [\sP^{\rm ss}(N_t)/\PGL(4)], \quad \fM(t):=\sP\sslash_{N_t} \SL(4).\]
\end{definition}

\begin{remark}\leavevmode
\begin{enumerate}
    \item Laza and O'Grady show that the VGIT quotients do not depend on choice of $\delta$, so the lack of $\delta$ in the notation is justified (see also Theorem \ref{thm:LOmain-alldeg}(1)). 
    \item Since $N_t$ is only semi-ample for $t  = \frac{1}{2}$, they define $\fM(\frac{1}{2})$ to be $\Proj R(\sP, N_{\frac{1}{2}})^{\SL(4)}$, and show this is isomorphic to $\Chow_{(2,4)}\sslash \SL(4)$. 
\end{enumerate} \end{remark} 

The following two results from \cite{LO} will be required to relate the VGIT moduli spaces and K-moduli spaces.

\begin{prop}{\cite[Proposition 5.4]{LO}}\label{prop:Linfinity}
Let $\chow: U \to \Chow_{(2,4)}$ be the Hilbert-Chow morphism and let $\overline{L}_{\infty}\in \Pic (\bP(E))_{\bQ}$ be the unique extension of $\chow^*L_\infty$ to $\bP(E)$.  Then, 
\[ \overline{L}_\infty = 4\eta + 2 \xi .\]
\end{prop}

\begin{lem}{\cite[Proposition 5.11]{LO}}\label{lem:GITssU}
For each $t\in (\delta,\frac{1}{2}] \cap \bQ$, the VGIT semistable locus $\sP^{\rm ss}(N_t)$ of slope $t$ is a Zariski open subset of $U$.
\end{lem}

We now state the main VGIT result of \cite{LO}, noting that their results also hold for the VGIT quotient stacks. Let $\Hilb_{(2,4)}$ denote the closure of $U$ inside the relevant Hilbert scheme, and let $L_m$ denote the Pl\"ucker line bundle corresponding to the $m$th Hilbert point.

\begin{theorem}\cite[Theorem 5.6]{LO}\label{thm:LOmain}
Let $\delta$ be as above. The following hold:
\begin{enumerate}
    \item For $t \in (\delta, \frac{1}{3})$,  the moduli space $\fM(t) \cong \bP(E) \sslash_{\eta + t\xi} \SL(4)$. 
    \item For $t \in (\delta, \frac{1}{6})$, we have $\fM(t) \cong \fM$. 
    \item For $m \geq 4$, we have $\Hilb_{(2,4)}\sslash_{L_m}  \SL(4) \cong \fM(t(m))$, where $t(m) = \dfrac{(m-3)^2}{2(m^2 - 4m + 5)}$. 
    \item $\fM(\frac{1}{2})\cong \Chow_{(2,4)}\sslash \SL(4)$.
\end{enumerate}
\end{theorem}


Before stating their main result, we review some results from VGIT.

\subsubsection{Variation of GIT} The general theory of Variation of GIT quotients (VGIT) can be found in \cite{Tha96, dolgachevhu}. The goal here is to compare $\fM(t)$ for $t \in (\delta, \frac{1}{2}) \cap \bQ$, in particular how varying the line bundle $N_t$ changes the GIT quotient.  The main results of VGIT state that this interval can be subdivided into finitely many open chambers, and on each open chamber the space $\fM(t)$ remains unchanged (\cite[Theorem 2.4]{Tha96} and \cite[Theorem 0.2.3]{dolgachevhu}). The finitely many values where the space $\fM(t)$ does change are called walls.  Here, there are birational morphisms $\fM(t-\epsilon) \to \fM(t) \leftarrow \fM(t+\epsilon)$, and there are additionally wall-crossing rational maps $\fM(t-\epsilon) \dashrightarrow \fM(t+\epsilon)$ (\cite[Theorem 3.3]{Tha96}).

Later on, we will need the following foundational results in VGIT, and we refer the reader to the survey \cite[Sections 3 and 4]{LazaGIT}, and the references therein.

\begin{lem}\label{lem:VGITbasics}
Let $(X,\cL_0)$ be a polarized projective variety. Let $G$ be a reductive group acting on $(X,\cL_0)$. Let $\cL$ be a $G$-linearized line bundle on $X$. For a rational number $0<\epsilon\ll 1$, consider the G-linearized ample $\bQ$-line bundle $\cL_{\pm}:= \cL_{0}\otimes \cL^{\otimes(\pm \epsilon)}$.
\begin{enumerate}
    \item Let $X \sslash_{\calL_0} G$ and $X \sslash_{\calL_{\pm}} G$ denote the VGIT quotients.  If $X^{\rm ss}(0)$ and $X^{\rm ss}(\pm)$ denote the respective VGIT semistable loci, then there are open inclusions $X^{\rm ss}(\pm) \subseteq X^{\rm ss}(0)$.
    \item  For any closed point $x\in X^{\rm ss}(0)\setminus X^{\rm ss}(\pm)$, there exists a $1$-PS $\sigma$ in $G$ such that
    \[
    \mu^{\cL_0}(x, \sigma)=0, \quad\textrm{and}\quad  \mu^{\cL_\pm}(x, \sigma)<0.  
    \]
\end{enumerate}

\end{lem}

\begin{proof}
(1) This is the well-known semi-continuity property of semistable loci from \cite[Theorem 4.1]{Tha96} and \cite[\S3.4]{dolgachevhu} (see also \cite[Lemma 3.10]{LazaGIT}).

(2) By symmetry we may assume that $x$ is VGIT unstable with respect to $\cL_+$. Hence by Hilbert-Mumford numerical criterion, there exists a $1$-PS $\sigma_0$ in $G$ such that $\mu^{\cL_+}(x, \sigma_0)<0$. Let $T$ be a maximal torus of $G$ containing $\sigma_0$. By \cite[Chapter 2, Proposition 2.14]{MFK94}, we know that there exist two rational piecewise linear function $h_0$ and $h$ on $\Hom_{\bQ}(\bG_m, T)$ such that for any $1$-PS $\lambda$ in $T$, we have
\[
\mu^{\cL_0}(x,\lambda)=h_0(\lambda), \quad \textrm{and} \quad \mu^{\cL}(x,\lambda)=h(\lambda).
\]
Since $x\in X^{\rm ss}(0)$, we know that $h_0(\lambda)\geq 0$ for any $\lambda\in \Hom_{\bQ}(\bG_m, T)$. On the other hand, $\mu^{\cL_+}(x,\sigma_0)=h_0(\sigma_0)+\epsilon h(\sigma_0)<0$. Hence there exists $\sigma\in \Hom_{\bQ}(\bG_m, T)$ such that $h_0(\sigma)=0$ and $h(\sigma)<0$. The proof is finished.
\end{proof}

Finally, we state the main result from \cite{LO}.

\begin{theorem}\cite[Theorem 1.1]{LO}\label{thm:LOwallcrossings}
Let $\beta\in [0,1]$, and let $t(\beta) = \dfrac{1}{4\beta +2} \in [\frac{1}{6}, \frac{1}{2}]$. The period map \[\mathfrak{p}: \fM \cong \sF(1)  \dashrightarrow \sF(0) \cong \sF^*\] is the composition of elementary birational maps with 8 critical values of $\beta$. Moreover, there is an isomorphism $\fM(t(\beta)) \cong \sF(\beta)$. In particular, the intermediate spaces are the VGIT quotients described above, and are related by elementary birational maps. Finally, the map $\sF(1/8) \to \sF(0) \cong \sF^*$ is the $\bQ$-Cartierization of $H_h$. 
\end{theorem}

\section{Degenerations of $\bP^1 \times \bP^1$ in K-moduli spaces}\label{sec:surfaces}

\subsection{K-moduli spaces of curves on $\bP^1 \times \bP^1$}

In this section, we will define the K-moduli spaces which generically parametrize smooth $(d,d)$-curves on $\bP^1\times\bP^1$.

\begin{prop}\label{prop:p1xp1kss} Let $d\geq 3$ be an integer. Let $C$ be a $(d,d)$-curve on $\bP^1\times\bP^1$. If $\lct(\bP^1\times\bP^1;C)>\frac{2}{d}$ (resp. $\geq \frac{2}{d}$), then the log Fano pair $(\bP^1 \times \bP^1, cC)$ is K-stable (resp. K-semistable) for any $c\in (0, \frac{2}{d})$. In particular, $(\bP^1\times\bP^1, cC)$ is K-stable  for any $c\in (0, \frac{2}{d})$ if either $C$ is smooth or $d=4$ and $C$ has at worst ADE singularities.
\end{prop}

\begin{proof} This follows from interpolation (see \cite[Proposition 2.13]{ADL} or \cite[Lemma 2.6]{Der16}), since the pair $(\bP^1 \times \bP^1, \frac{2}{d}C)$ is klt (resp. lc) and $\bP^1\times\bP^1$ is K-polystable. \end{proof}

 We begin to define the K-moduli stack $\oMK_{d,c}$ and the K-moduli space $\overline{K}_{d,c}$.
 Let $\chi_0(\cdot)$ be the Hilbert polynomial of the polarized Fano manifold $(\bP^1\times\bP^1,-K_{\bP^1\times\bP^1})$, i.e. $\chi_0(m)= 4m^2 +4m+1$. Consider the K-moduli stack $\cK\cM_{\chi_0, d/2, c}$ and K-moduli space $KM_{\chi_0, d/2, c}$ where $d\geq 3$ is an integer and $c\in (0, \frac{2}{d})\cap\bQ$.
 
\begin{prop}\label{prop:modconnected}
Let $d\geq 3$ be an integer.
The K-moduli stack $\cK\cM_{\chi_0, d/2, c}$ and K-moduli space $KM_{\chi_0, d/2, c}$ are both normal. Moreover, we have the following cases.
\begin{enumerate}
    \item If $d$ is odd, then $\cK\cM_{\chi_0, d/2, c}$ is connected and generically parametrizes $(\bP^1\times\bP^1, cC)$ where $C\in |\cO_{\bP^1\times\bP^1}(d,d)|$ is a smooth curve.
    \item If $d$ is even, then $\cK\cM_{\chi_0, d/2, c}$  has at most two connected components. One of these components generically parametrizes $(\bP^1\times\bP^1, cC)$ where $C\in |\cO_{\bP^1\times\bP^1}(d,d)|$ is a smooth curve; the other component, if it exists, generically parametrizes $(\bF_1, cC')$ where $C'\in |\cO_{\bF_1}(-\frac{d}{2}K_{\bF_1})|$ is a smooth curve on the Hirzebruch surface $\bF_1$.
\end{enumerate}
\end{prop}

\begin{proof}
The normality of $\cK\cM_{\chi_0, d/2, c}$ and $KM_{\chi_0, d/2, c}$ is a direct consequence of Theorem \ref{thm:modnormal}. For the rest, notice that there are only two smooth del Pezzo surfaces of degree $8$ up to isomorphism: $\bP^1\times\bP^1$ and $\bF_1$. In addition, they are not homeomorphic since their intersection pairings on $H^2(\cdot, \bZ)$ are not isomorphic. By Proposition \ref{prop:p1xp1kss} we know that $(\bP^1\times\bP^1,cC)$ where $C$ is a smooth $(d,d)$-curve is always parametrized by $\cK\cM_{\chi_0, d/2, c}$. 
If $d$ is odd, then $-\frac{d}{2}K_{\bF_1}$ is not represented by any Weil divisor since it has fractional intersection with the $(-1)$-curve on $\bF_1$. Hence $\bF_1$ will not appear in $\cK\cM_{\chi_0, d/2, c}$ when $d$ is odd. The proof is finished.
\end{proof}


\begin{defn}\label{defn:modulispace}
Let $d\geq 3$ be an integer.
For $c\in(0, \frac{2}{d})\cap \bQ$, let $\oMK_{d,c}$ denote the connected component of $\cK\cM_{\chi_0, d/2, c}$ where a general point parametrizes $(\bP^1\times\bP^1, cC)$ where $C\in |\cO_{\bP^1\times\bP^1}(d,d)|$ is a smooth curve. In other words, $\oMK_{d,c}$ is the moduli stack parametrizing
K-semistable log Fano pairs $(X,cD)$, where $X$ admits a $\bQ$-Gorenstein smoothing to $\bP^1 \times \bP^1$ and the effective $\bQ$-Cartier Weil divisor $D \sim_{\bQ} -\frac{d}{2}K_X$. We let $\overline{K}_{d,c}$ denote the good moduli space of $\oMK_{d,c}$. From  Theorems  \ref{thm:projectivity}, \ref{thm:modnormal}, and Proposition \ref{prop:modconnected} we know that $\oMK_{d,c}$ is a connected smooth Artin stack of finite type over $\bC$, and $\oK_{d,c}$ is a  normal projective variety over $\bC$.

\end{defn}

The following theorem is a direct consequence of \cite[Theorem 1.2]{ADL} and Proposition \ref{prop:p1xp1kss}.

\begin{thm}\label{thm:generalwall}
Let $d\geq 3$ be an integer. 
There exist rational numbers 
\[
0=c_0 <c_1<c_2<\cdots< c_k =\frac{2}{d}
\]
such that for each $0\leq i\leq k-1$ the K-moduli stacks $\oMK_{d,c}$ are independent of the choice of $c\in (c_i,c_{i+1})$. For each $1\leq i\leq k-1$ and $0<\epsilon\ll 1$, we have open immersions 
\[
\oMK_{d,c_i-\epsilon}\hookrightarrow \oMK_{d, c_i}\hookleftarrow \oMK_{d,c_i+\epsilon}
\]
which induce projective birational morphisms
\[
\oK_{d,c_i-\epsilon}\rightarrow \oK_{d, c_i}\leftarrow \oK_{d,c_i+\epsilon}.
\]
Moreover, all the above morphisms have local VGIT presentations as in \cite[(1.2)]{AFS17}.
\end{thm}

In this paper, we are mainly interested in the case when $d=4$, although some results for general $d$ are presented in Section \ref{sec:generaldegree}. We always abbreviate $\oMK_{4,c}$ and $\oK_{4,c}$ to $\oMK_{c}$ and $\oK_{c}$, respectively.

\subsection{Classification of degenerations of $\bP^1 \times \bP^1$}
The goal of this section is to prove Theorem \ref{thm:surfaces}, which states that if $(X,cD)$ is a pair parametrized by $\oMK_{c}$ for some $c \in (0, \frac{1}{2})$, then $X$ is isomorphic to either $\bP^1 \times \bP^1$ or $\bP(1,1,2)$. Later on, we will show (in Theorem \ref{thm:surfacesalld}) that the same is true in $\oMK_{d,c}$ for $0 < c < \frac{4-\sqrt{2}}{2d}$ and $d \geq 3$.  First we show that if $X$ is a normal $\bQ$-Gorenstein deformation of $\bP^1 \times \bP^1$, then $\rho(X) \leq 2$. 

\begin{prop}\label{prop:rho} Let $X$ be a log del Pezzo surface. Suppose that $X$ admits a $\bQ$-Gorenstein deformation to $\bP^1 \times \bP^1$. Then $\rho(X) \leq 2$. \end{prop}

\begin{proof} 
Let $\cX\to T$ be a $\bQ$-Gorenstein smoothing of $X$, i.e. $0\in T$ is a smooth germ of pointed curve, $\cX_0\cong X$, and $\cX_t\cong \bP^1\times\bP^1$ for $t\in T\setminus \{0\}$. By passing to a finite cover of $0\in T$, we may assume that $\cX^\circ\cong (\bP^1\times\bP^1)\times T^\circ$ where $\cX^\circ:=\cX\setminus \cX_0$ and $T^\circ := T \setminus \{0\}$. 
First using \cite[Lemma 2.11]{Hac04}, we show that $\mathrm{Cl}(\calX) \cong \mathbb{Z}^{2}$. Indeed,  consider the exact sequence 
\[ 0 \to \mathbb{Z}X \to \mathbb{Z} X \to \mathrm{Cl}(\calX)\to \mathrm{Cl}(\calX^\circ) \to 0,\] which gives $\mathrm{Cl}(\calX) \cong \mathrm{Cl}(\calX^\circ)\cong \mathbb{Z}^{2}$.

Now we follow the proof of \cite[Proposition 6.3]{Hac04}. First note that there is an isomorphism $\Pic(\calX) \to \Pic(X)$, and so we obtain the inequality: 
\[ \rho(X) = \dim \Pic(X) \otimes \mathbb{Q} = \dim \Pic(\calX) \otimes \mathbb{Q} \leq \dim \mathrm{Cl}(\calX) \otimes \mathbb{Q} = 2,\]
with equality if and only if $\calX$ is $\mathbb{Q}$-factorial. \end{proof}

A result of Hacking-Prokhorov now classifies the possible $\bQ$-Gorenstein smoothings of $\bP^1 \times \bP^1$ (see \cite[Theorem 1.2]{hparxiv} and \cite[Proposition 2.6]{hp}).

\begin{prop}[Hacking-Prokhorov]\label{prop:rhoindex}
Let $X$ be a log del Pezzo surface admitting a $\bQ$-Gorenstein smoothing to $\bP^1\times\bP^1$.  There are two cases.
\begin{enumerate}
    \item If $\rho(X)=1$, then $X$ is a $\bQ$-Gorenstein partial smoothing of a weighted projective plane $\bP(a^2,b^2,2c^2)$ where $(a,b,c)\in\bZ_{>0}^3$ subject to the equation
    \[
    a^2+b^2+2c^2=4abc.
    \]
    In particular, the local index  $\ind(x,K_X)$ is odd for any $x\in X$.
    \item If $\rho(X)=2$, then $X$ only has quotient singularities of type
    $\frac{1}{n^2}(1,an-1)$ where $\gcd(a,n)=1$.
\end{enumerate}
\end{prop}

Suppose $x\in X$ is a surface $T$-singularity. We denote by $\mu_x$ the \emph{Milnor number} of a $\bQ$-Gorenstein smoothing of $x\in X$. If $x\in X$ is a cyclic quotient $T$-singularity of type $\frac{1}{en^2}(1,ena-1)$, then $\mu_x=e-1$.

\begin{theorem}\label{thm:indexbound} Let $(X,cD)
 $ be a K-semistable log Fano pair that admits a $\bQ$-Gorenstein smoothing
 to $(\bP^1\times\bP^1, cC_t)$ with $c\in (0,\frac{2}{d})$ and $C_t$ a curve of bidgree $(d,d)$.
 Let $x\in X$ be any singular point.
 \begin{enumerate}
     \item If $d$ is even or $\ind(x,K_X)$ is odd, then 
      \[
  \ind(x,K_X)\leq\begin{cases}
                  \min\{\lfloor\frac{3}{\sqrt{2}(2-cd)}\rfloor,d+1\} & \textrm{ if }\mu_x=0,\\
                  \min\{\lfloor\frac{3}{2(2-cd)}\rfloor,d\} & \textrm{ if }\mu_x=1. 
                 \end{cases}
 \]
 \item If $d$ is odd and $\ind(x,K_X)$ is even, then $\rho(X)=2$, $\mu_x=0$, and 
 \[
  \ind(x,K_X)\leq \min\{2\lfloor\tfrac{3}{2\sqrt{2}(2-cd)}\rfloor,2d-2\}.
  \]
  \end{enumerate}
\end{theorem}

\begin{proof}
 Let $\beta:=1-cd/2\in (0,1)$. 
 We know that an index $n$ point $x\in X$ is a 
 cyclic quotient singularity of type $\frac{1}{n^2}(1,na-1)$ or $\frac{1}{2n^2}(1, 2na-1)$ where $\gcd(a,n)=1$. If $\mu_x=0$, then the orbifold group of $x\in X$ has order $n^2$ which implies that $\hvol(x,X)=\frac{4}{n^2}$ by \cite[Proposition 4.10]{LL16}. Hence Theorem \ref{thm:local-vol-global} implies that 
 \[
 8\beta^2 =(-K_X-cD)^2 \leq \frac{9}{4}\hvol(x,X)=\frac{9}{n^2}.
 \]
 This shows that $n\leq \frac{3}{2\sqrt{2} \beta}=\frac{3}{\sqrt{2}(2-cd)}$. Similarly, if $\mu_x=1$, then $x\in X$ has orbifold group of order $2n^2$ which implies that $n\leq \frac{3}{4 \beta}=\frac{3}{2(2-cd)}$. Hence the first terms in the index upper bounds are verified.
 
 The rest of this proof is devoted to verifying the second terms in the index upper bounds. 
 We know that $dK_X+2D\sim 0$ when $d$ is odd and $\frac{d}{2}K_X+D\sim 0$ when $d$ is even. If $x\not\in D$, then  $n\mid d$ hence $n\leq d$ (in fact $n\leq \frac d{2}$ if
 $d$ is even). Hence the second terms are verified for $x\not\in D$.

 From now on let us assume
 $x\in D$. Let $(\tilde{x}\in \widetilde{X})$ be the 
 smooth cover of $(x\in X)$, with $\widetilde{D}$ being the 
 preimage of $D$.  Assume $\tilde{x}\in\widetilde{X}$ has local coordinates  $(u,v)$ where the cyclic group action is scaling on 
  each coordinate. Let $u^i v^j$ be a monomial appearing
  in the equation on $\widetilde{D}$ with minimum $i+j=\ord_{\tilde{x}}\widetilde{D}$.
 
 \textbf{Case 1.} Assume $d$ is even and $\mu_x=0$. Then the orbifold group of $x\in X$ has order $n^2$. 
 Since the finite degree formula
 is true in dimension $2$ by \cite[Theorem 4.15]{LLX18}, we have 
 \[
  \hvol(\tilde{x},\widetilde{X},c\widetilde{D})
  =n^2\cdot\hvol(x,X,c D).
  \]
  On the other hand, Theorem \ref{thm:local-vol-global} implies that
  \[
  8\beta^2=(-K_X-cD)^2\leq \frac{9}{4}\hvol(x, X, cD)=\frac{9}{4n^2}\hvol(\tilde{x},\widetilde{X},c\widetilde{D}).
  \]
  So we have
  \begin{equation}\label{eq:index1}
   n\leq \frac{3\sqrt{  \hvol(\tilde{x},\widetilde{X},c\widetilde{D})}}{4\sqrt{2}\beta}
   \leq \frac{3(2-c~\ord_{\tilde{x}}\widetilde{D})}{4\sqrt{2}\beta}.
  \end{equation}
  In particular we have $n<\frac{3}{2\sqrt{2}\beta} $.
  We know that $\lct_{\tilde{x}}(\widetilde{X};\widetilde{D})>
  c$, and Skoda \cite{skoda} implies $\lct_{\tilde{x}}(\widetilde{X};\widetilde{D})\leq
  \frac{2}{\ord_{\tilde{x}}\widetilde{D}}$, so we have
  $  \ord_{\tilde{x}}\widetilde{D}<\frac{2}{c}$.
  Since $\frac{d}{2}K_X+D\sim 0$, we have $i+(na-1)j\equiv \frac{d}{2}na\mod n^2$ which implies $i\equiv j\mod n$.
  
  If $\beta\geq \frac{3}{2\sqrt{2}d+3}$
  then $n<\frac{3}{2\sqrt{2}\beta}\leq d+\frac{3}{2\sqrt{2}}$ which implies $n\leq d+1$. Thus we may assume
  $\beta<\frac{3}{2\sqrt{2}d+3}$. Then 
  \[
   i+j=  \ord_{\tilde{x}}\widetilde{D}<\frac{2}{c}<d+\frac{3}{2\sqrt{2}}.
  \]
  Hence $i+j\leq d+1$. Assume to the contrary that $n\geq d+2$. Then $i\equiv j\mod
  n$ and $i+j<n$ implies that $i=j$. Hence $i+(na-1)j\equiv \frac{d}{2}na\mod n^2$ implies $i\equiv \frac{d}{2}\mod n$. But since $i\leq \frac{d+1}{2}<n$,
  we know that $i=j=\frac{d}{2}$. Then \eqref{eq:index1} implies that 
  \[
  n\leq \frac{3(2-c(i+j))}{4\sqrt{2}\beta}=\frac{6\beta}{4\sqrt{2}\beta}<2.
  \]
  We reach a contradiction. 
  
  \textbf{Case 2.} Assume $d$ is even and $\mu_x=1$.  Then the orbifold group of $x\in X$ has order $n^2$. By a similar argument as in Case 1, we know that 
  \[
  8\beta^2=(-K_X-cD)^2\leq \frac{9}{4}\hvol(x, X, cD)=\frac{9}{8n^2}\hvol(\tilde{x},\widetilde{X},c\widetilde{D}).
  \]
  Hence
  \begin{equation}\label{eq:index2}
  n\leq \frac{3\sqrt{  \hvol(\tilde{x},\widetilde{X},c\widetilde{D})}}{8\beta}
   \leq \frac{3(2-c~\ord_{\tilde{x}}\widetilde{D})}{8\beta}.    
  \end{equation}
  In particular we have $n<\frac{3}{4\beta}$.
  
  If $\beta\geq \frac{3}{4d+3}$
  then $n<\frac{3}{4\beta}\leq d+\frac{3}{4}$ which implies $n\leq d$. Thus we may
  assume $\beta<\frac{3}{4d+3}$. Then
  \[
      i+j=  \ord_{\tilde{x}}\widetilde{D}<\frac{2}{c}
      <d+\frac{3}{4}.  
  \]
  Hence $i+j\leq d$. Assume to the contrary that $n\geq d+1$.
  Then $i\equiv j\mod n$ and $i+j<n$ implies $i=j$.
  Hence $i+(na-1)j\equiv \frac{d}{2}na\mod n^2$ implies $i\equiv \frac{d}{2}\mod n$. But since $i\leq \frac{d}{2}<n$,
  we know that $i=j=\frac{d}{2}$. Then \eqref{eq:index2} implies that 
  \[
  n\leq \frac{3(2-c(i+j))}{8\beta}=\frac{6\beta}{8\beta}<1.
  \]
  We reach a contradiction. 
  
  \textbf{Case 3.} Assume $d$ is odd and $\mu_x=0$. In this case we have $dK_X+2D\sim 0$ which implies $2(i+(na-1)j)\equiv dna \mod n^2$. 
  If $n$ is odd, then clearly $i\equiv j\mod n$. By the same argument as Case 1, we know $i=j=\frac{d}{2}$ if $n\geq d+2$, hence a contradiction.
  
  If $n$ is even, then we do a finer analysis. Since both $d$ and $a$ are odd, from $2(i+(na-1)j)\equiv dna \mod n^2$ we know that $i-j\equiv \frac{n}{2}\mod n$. Thus 
  $n\leq 2(i+j)<\frac{4}{c}=\frac{2d}{1-\beta}$.  Besides, \eqref{eq:index1} implies that $n<\frac{3}{2\sqrt{2}\beta}$. Hence
  \[
  n<\min\left\{\frac{2d}{1-\beta},\frac{3}{2\sqrt{2}\beta}\right\}\leq \frac{2\sqrt{2}(2d)+3}{2\sqrt{2}(1-\beta)+2\sqrt{2}\beta}=2d+\frac{3}{2\sqrt{2}}.
  \]
  Thus $n\leq 2d$. Assume to the contrary that $n=2d$, then $i+j\geq \frac{n}{2}=d$. Hence \eqref{eq:index1} implies that 
  \[
  2d=n\leq \frac{3(2-c(i+j))}{4\sqrt{2}\beta}\leq \frac{3(2-cd)}{4\sqrt{2}\beta}=\frac{3}{2\sqrt{2}}.
  \]
  We reach a contradiction. Thus we have $n\leq 2d-2$.
  
  \textbf{Case 4.} Assume $d$ is odd and $\mu_x=1$. Then by \cite[Proposition 2.6]{hp}, we know that $\rho(X)=1$. So $n$ is odd by Proposition \ref{prop:rhoindex}. Hence $2(i+(2na-1)j)\equiv dna \mod n^2$ implies $i\equiv j\mod n$. By a similar argument as in Case 2, we know $i=j=\frac{d}{2}$ if $n\geq d+1$, hence a contradiction. 
\end{proof}

%

%
%
%

The index bounds in Theorem \ref{thm:indexbound} allow us to limit the surfaces that appear in pairs parameterized by the moduli stack $\oMK_{c}$.

\begin{theorem}\label{thm:surfaces}
Let $(X,cD)$ be a K-semistable log Fano pair that admits a $\bQ$-Gorenstein smoothing to $(\bP^1\times\bP^1, cC_t)$ with $c\in (0,\frac{1}{2})$ and $C_t$ a $(4,4)$ curve.
Then, $X$ must be isomorphic to either $\bP^1 \times \bP^1$ or $\bP(1,1,2)$. \end{theorem}

\begin{proof} By Proposition \ref{prop:rho}, we know that $\rho(X) \leq 2$. We start with $\rho(X) = 1$. In this case, by Proposition \ref{prop:rhoindex}, we know that $X$ is a weighted projective space of the form $\bP(a^2, b^2, 2c^2)$ where $a^2 + b^2 + 2c^2 = 4abc$, or a partial smoothing. 
We begin enumerating the possible integer solutions and see that the first few are \[( a,b,c) = (1,1,1), (1,3, 1), (1, 3, 5), (11, 3, 5), \dots.\] We can exclude the last 2 (and any with higher index) by the index bound of Theorem \ref{thm:indexbound}. The first gives $\bP(1,1,2)$ and the second gives $\bP(1,2,9)$. We now show that the singularity $\frac{1}{9}(1,2)$ cannot appear. 

Assume to the contrary that $x\in X$ is of type $\frac{1}{9}(1,2)$. Suppose $D \sim -2K_X$ and  consider a smooth covering $(\tilde{x} \in \widetilde{X}) \to (x \in X)$. Note that we may assume $x \in D$, because otherwise if $x \notin D$ then $\ind(x,K_X) \leq 2$, and we obtain a contradiction. Consider local coordinates of $\tilde{x} \in \widetilde{X}$ namely $(u,v)$. Let $u^i v^j$ be a monomial appearing  in the equation on $\widetilde{D}$ with minimum $i+j=\ord_{\tilde{x}}\widetilde{D}$. Then $i + 2j \equiv 6 \mod 9$. Since we know that $(X, cD)$ is klt at $x$, we have that \[ \frac{2}{i+j} \geq \lct(\widetilde{D}) > c\] and so in particular $i + j < \frac{2}{c}$. 
By \eqref{eq:index1} with $n=3$ and $\beta=1-2c$, we have 
 \[2-(i+j)c \geq 4\sqrt{2}(1-2c).\] Since this inequality holds for some $0 < c < \frac{1}{2}$, we have $i+j\leq 3$ because otherwise
\[
2-(i+j)c\leq 2-4c<4\sqrt{2}(1-2c)
\]
which contradicts the previous inequality. Putting this together with $i + 2j = 6 \mod 9$, we see that $(i, j) = (0,3)$. 

Consider the valuation $w$ on $\widetilde{X}$ which is the monomial valuation in the coordinates $(u,v)$ of weights $(1,2)$. In particular $w(\widetilde{D}) = 6$. Moreover, $A_{\widetilde{X}}(w) = 3$ and $\vol(w) = \frac{1}{2}$. Then we note that 
\[
\hvol(\tilde{x}, \widetilde{X}, c\widetilde{D}) \leq (A_{\widetilde{X}}(w) - c~w(\widetilde{D}))^2\vol(w) = \frac{(3-6c)^2}{2}.
\] By \eqref{eq:index1} we have 
\[
4\sqrt{2}(1-2c)\leq \sqrt{\hvol(\tilde{x}, \widetilde{X}, c\widetilde{D})}\leq \frac{3-6c}{\sqrt{2}}
\]
which gives $4\sqrt{2} \leq \frac{3}{\sqrt{2}}$, a contradiction. Thus the surface $X$ with a $\frac{1}{9}(1,2)$ singularity cannot appear. In particular, the only surface with $\rho(X) = 1$ is $X \cong \bP(1,1,2)$.

Now we consider $\rho(X) = 2$. By Proposition \ref{prop:rhoindex}, we know that the only singular points of $X$ are of the form $\frac{1}{n^2}(1, na-1)$ with $n\leq 5$. We already excluded $\frac{1}{9}(1,2)$ so we only need to consider $n = 2, 4, 5$. 

Let us consider $n=4$, namely a  singularity of type $\frac{1}{16}(1, 3)$. We show that this singularity cannot occur.
As before, consider a smooth covering $(\tilde{x} \in \widetilde{X}) \to (x \in X)$ and suppose $D \sim -2K_X$. Note that we may assume $x \in D$, because otherwise if $x \notin D$ then $\ind(x, K_X) \leq 2$, and we obtain a contradiction. Consider local coordinates of $\tilde{x} \in \widetilde{X}$ namely $(u,v)$.
Let $u^i v^j$ be a monomial appearing  in the equation on $\widetilde{D}$ with minimum $i+j=\ord_{\tilde{x}}\widetilde{D}$.
Then $i + 3j \equiv 8 \mod 16$, and $i + j < \frac{2}{c}$. By \eqref{eq:index1} with $n=4$ and $\beta=1-2c$, we have
\[ 4 \sqrt{2} (1-2c) \leq (2-c(i+j)).\]
Since this inequality holds for some $0 < c < \frac{1}{2}$, we have $i+j\leq 3$ by the same reason in $n=3$. This contradicts with $i+3j\equiv 8\mod 16$. In particular, a singularity of type $\frac{1}{16}(1,3)$ cannot occur.

Next let us consider $n=5$, namely a singularity of type $\frac{1}{25}(1,4)$ or $\frac{1}{25}(1,9)$. We again show that these singularities cannot occur. With the same set up as the previous paragraph, we have either $i+4j\equiv 10 \mod 25$ or $i+9j\equiv 20 \mod 25$. Moreover, we again have $i+j\leq 3$ by the same reason in $n=3,4$ but this contradicts to the congruence equations. Therefore, a singularity of type $\frac{1}{25}(1,4)$ or $\frac{1}{25}(1,9)$ cannot occur.

After the above discussions, the only case left to study is $\rho(X)=2$ and $X$ has only singularities of type $\frac{1}{4}(1,1)$. If $X$ is singular, then by \cite[Table 6 and Theorem 7.15]{nakayama} (see also \cite{AN}), we know that $X$ is isomorphic to a blow up of $\bP(1,1,4)$ at a smooth point. However, in this case $X$ admits a $\bQ$-Gorenstein smoothing to the Hirzebruch surface $\bF_1$ which is not homeomorphic to $\bP^1\times\bP^1$. This is a contradiction. Hence $X$ is smooth and isomorphic to $\bP^1\times\bP^1$.
\end{proof}

\begin{remark}\label{rmk:oursurfaces}
Let $(X,cD)$ be a K-semistable log Fano pair that admits a $\bQ$-Gorenstein smoothing to $(\bP^1\times\bP^1, cC_t)$ with $c\in (0,\frac{1}{2})$ and $C_t$ a $(4,4)$ curve.  By Theorem \ref{thm:surfaces}, this implies that $X$ is either $\bP^1 \times \bP^1$ or $\bP(1,1,2)$.  Therefore, there exists a closed embedding $(X,D)\hookrightarrow\bP^3$ such that $X \in |\cO_{\bP^3}(2)|$ and $D\sim -2K_X$ are $(2,4)$ complete intersections inside $\bP^3$.  Hence, all K-semistable pairs $(X, cD)$ with $c \in (0,\frac{1}{2})$ are parametrized by a Zariski open subset of $U$. \end{remark}

\begin{thm}\label{thm:surfacesalld}
Let $(X,cD)$ be a K-semistable log Fano pair that admits a $\bQ$-Gorenstein smoothing to $(\bP^1\times\bP^1, cC_t)$ with $c\in (0,\frac{4-\sqrt{2}}{2d})$ and $C_t$ a $(d,d)$ curve where $d \geq 3$.
Then, $X$ must be either $\bP^1 \times \bP^1$ or $\bP(1,1,2)$.
\end{thm}

\begin{proof}
By Proposition \ref{prop:rho}, $\rho(X) \leq 2$. By the index bound of Theorem \ref{thm:indexbound}, for $ c < \frac{4 - \sqrt{2}}{2d}$ we know that  $\ind(x, K_X) < 3$.
If $\rho(X)=1$, then by Proposition \ref{prop:rhoindex} we know that $X$ is Gorenstein which implies that $X\cong \bP(1,1,2)$. If $\rho(X)=2$, then by Proposition \ref{prop:rhoindex} we know that either $X$ is smooth hence isomorphic to $\bP^1\times\bP^1$, or $X$ has only singularities of type $\frac{1}{4}(1,1)$. The latter case cannot happen by the end of the proof of Theorem \ref{thm:surfaces}.
Therefore, the only surfaces appearing are $\bP^1 \times \bP^1$ and $\bP(1,1,2)$.
\end{proof}

\section{Wall crossings for K-moduli and GIT}\label{sec:main}


In this section, we prove Theorem \ref{mthm:thmintro}, that is, for $0<c < \frac {1}{2}$, the K-moduli stack $\oMK_{c}$ coincides with the GIT moduli stack $\sM(t)$ with $t=\frac{3c}{2c+2}$ (see Definition \ref{def:VGIT}).  The important observation comes from Theorem \ref{thm:surfaces}: the surfaces $X$ in the pairs parametrized by $\oMK_{c}$ are $\bP^1 \times \bP^1$ or $\bP(1,1,2)$ which are quadric surfaces in $\bP^3$, and the divisors $D$ can therefore be viewed as $(2,4)$-complete intersections in $\bP^3$.

\subsection{The first wall crossing}

In this section, we show that GIT-(poly/semi)stability of $(4,4)$-curves on $\bP^1\times\bP^1$ and $c$-K-(poly/semi)stability coincide for $c < \frac 1{8}$. Moreover, we show that $c_1=\frac{1}{8}$ is the first wall for K-moduli stacks $\oMK_c$.

\begin{defn}
A $(4,4)$-curve $C$ on $\bP^1\times\bP^1$ gives a point $[C] \in \bfP_{4,4}:= \bP(H^0(\bP^1\times\bP^1, \cO(4,4)))$.  We say $C$ is \emph{GIT (poly/semi)stable} if $[C]$ is GIT (poly/semi)stable with respect to the natural $\Aut(\bP^1\times\bP^1)$-action on  $(\bfP_{4,4}, \cO(2))$. We define the \emph{GIT quotient stack} $\sM$ and the \emph{GIT quotient space} $\fM$ as
\[
\sM:=[\bfP_{4,4}^{\rm ss}/\Aut(\bP^1\times\bP^1)], \qquad 
\fM:=\bfP_{4,4}^{\rm ss}\sslash \Aut(\bP^1\times\bP^1).
\]
\end{defn}

\begin{theorem}\label{thm:firstwall} For any $0 < c < \frac 1{8}$, a curve $C \subset \bP^1\times\bP^1$ of bidgree $(4,4)$ is GIT-(poly/semi)stable if and only if the log Fano pair $(\bP^1 \times \bP^1, cC)$ is K-(poly/semi)stable. Moreover, there is an isomorphism of Artin stacks $\oMK_{c} \cong \sM$. \end{theorem}

\begin{proof}
We first show that K-(poly/semi)stability of $(\bP^1\times\bP^1, cC)$ implies GIT (poly/semi)stability of $C$ for any $c\in (0,\frac{1}{2})$. Consider the universal family $\pi: (\bP^1 \times \bP^1 \times \mathbf{P}_{(4,4)}, c\calC) \to \mathbf{P}_{(4,4)}$ over the parameter space of $(4,4)$-curves on $\bP^1\times\bP^1$. It is clear that $\calC \in |\cO(4,4,1)|$. Hence by Proposition \ref{prop:logCM2} we have
\begin{align*}
\lambda_{\CM, \pi, c\calC}&  = -\pi_* (-K_{\bP^1 \times \bP^1 \times \mathbf{P}_{(4,4)}/\bfP_{(4,4)}}-c\calC)^3
= -\pi_* (\cO(2-4c,2-4c,-c))^3 \\
& = -3(\cO_{\bP^1\times\bP^1}(2-4c,2-4c)^2) \cO_{\bfP_{(4,4)}}(-c)
 = \cO_{\bfP_{(4,4)}}(3(2-4c)^2 c).
\end{align*}
Hence the CM line bundle $\lambda_{\CM, \pi, c\calC}$ is ample whenever $c\in (0, \frac{1}{2})$. Hence the statement of K implying GIT directly follows from Theorem \ref{thm:paultian}.

Next we show the converse, i.e. GIT-(poly/semi)stability of $C$ implies K-(poly/semi)stability of $(\bP^1\times\bP^1,cC)$ for $c< \frac{1}{8}$. Indeed, using similar argument as the proof of \cite[Theorem 5.2]{ADL} with a key ingredient from properness of K-moduli spaces, it suffices to show that any pair $(X,D)$ appearing in the K-moduli stack $\ocK_c$ for $c < \frac{1}{8}$ satisfies that $X\cong\bP^1\times\bP^1$ and $D$ is a $(4,4)$-curve. Since $\bP^1\times\bP^1$ has no non-trivial smooth degeneration, it suffices to show that $X$ is smooth. Assume to the contrary that $X$ is singular at a point $x\in X$. Then by \cite{LL16} we know that 
\[
 8(1-2c)^2=(-K_X-cD)^2\leq \frac{9}{4}\hvol(x,X,cD)\leq \frac{9}{4}\hvol(x,X)\leq \frac{9}{2}.
\]
This implies that $c\geq \frac{1}{8}$ which is a contradiction.
Hence, for $c < \frac{1}{8}$, a K-semistable pair $(X,cD)$ must be isomorphic to $(\bP^1 \times \bP^1, cC)$, where $C$ is a $(4,4)$-curve.


Summing up, the equivalence of K-(poly/semi)stability with GIT (poly/semi)stability yields a morphism $\phi: \sM\to \oMK_{c}$ which descends to an isomorphism $\fM\xrightarrow{\cong}\oK_c$. 
To conclude, it suffices to show that $\phi$ is an isomorphism between Artin stacks. The proof is similar to \cite[Theorem 3.24]{ADL}. Denote by $T:=\bfP_{4,4}^{\rm ss}$. Let $\pi:(\cX,\cD)\to T$ be the universal family. Recall from \cite[Section 3.1]{ADL} and Theorem \ref{thm:modnormal} that $\ocK_c\cong [Z_{c}^\circ/\PGL(N_m+1)]$ where $Z_{c}^\circ$ is the K-semistable locus in the Hilbert scheme of embedded by $m$-th multiple of anti-canonical divisors. Denote by $\pi': (\cX',\cD')\to T'$ the universal family over $T':=Z_c^\circ$. Let $P$ be the $\PGL(N_m+1)$-torsor over $T$ induced from the vector bundle $\pi_*\cO_{\cX}(-mK_{\cX/T})$. Then from \cite[Proof of Theorem 3.24]{ADL} we see that there is an $\Aut(\bP^1\times\bP^1)$-equivariant morphism $\psi: P\to T'$ whose descent is precisely $\phi$. Hence in order to show $\phi$ is isomorphic it suffices to show that $\psi$ provides an $\Aut(\bP^1\times\bP^1)$-torsor. Indeed, since $\pi':\cX'\to T'$ is isotrivial where all fibers are isomorphic to $\bP^1\times\bP^1$, we may find an \'etale covering $\cup_i V_i\twoheadrightarrow T'$ such that there is an isomorphism $\rho_i: \cX'\times_{T'} V_i\xrightarrow{\cong} (\bP^1\times\bP^1)\times V_i$. Hence by pushing forward $(\cX',\cD')\times_{T'} V_i$ and its natural frame from $\bP^{N_m+1}$ to $(\bP^1\times\bP^1)\times V_i$ under $\rho_i$, we obtain a section $V_i\to P\times_{T'}V_i$ of $\psi\times_{T'}V_i$ which trivializes $\psi$. Thus the proof is finished.
\end{proof}


The following proposition shows that $c_1=\frac{1}{8}$ is the first wall of the K-moduli stacks $\oMK_c$. Note that it is also proved by Fujita \cite{Fuj20} independently using different methods.

\begin{prop}\label{prop:firstwallreplace}
Let $C=4H$ where $H$ is a smooth $(1,1)$-curve on $\bP^1\times\bP^1$. Let $c\in (0,\frac{1}{2})$ be a rational number. Then $(\bP^1\times\bP^1,cC)$ is K-semistable (resp. K-polystable) if and only if $c\leq \frac{1}{8}$ (resp. $<\frac{1}{8}$).  Moreover, the K-polystable degeneration of $(\bP^1\times\bP^1,\frac{1}{8}C)$ is isomorphic to $(\bP(1,1,2), \frac{1}{8}C_0)$ where $C_0=4 H_0$ and $H_0$ is the section at infinity.
\end{prop}

\begin{proof}
We first show that $(\bP^1\times\bP^1,\frac{1}{8}C)$ is K-semistable where $(\bP(1,1,2), \frac{1}{8}C_0)$ is its K-polystable degeneration. Choose an embedding $\bP^1\times\bP^1\hookrightarrow\bP^3$ as a smooth quadric surface. Then $H$ is a hyperplane section of $\bP^1\times\bP^1$. Pick projective coordinates $[x_0,x_1,x_2,x_3]$ of $\bP^3$ such that the hyperplane section through $H$ is given by $x_3=0$. Then the $1$-PS $\sigma: \bG_m \to \PGL(4)$ given by $\sigma(t)[x_0,x_1,x_2,x_3]=[tx_0,tx_1,tx_2,x_3]$ provides a special test configuration of $(\bP^1\times\bP^1, \frac{1}{2}H)$ whose central fiber is an ordinary quadric cone with a section at infinity of coefficient $\frac{1}{2}$, i.e. isomorphic to $(\bP(1,1,2), \frac{1}{2}H_0)$. By \cite{LL16} we know that $(\bP(1,1,2), \frac{1}{2}H_0)$ admits a conical K\"ahler-Einstein metric hence is K-polystable. The K-semistability of $(\bP^1\times\bP^1,\frac{1}{8}C)$ follows from openness of K-semistability \cite{BLX19, Xu19}. 

Next we show that $(\bP^1\times \bP^1, cC)$ is K-polystable for $c\in (0,\frac{1}{8})$. Clearly, it is K-semistable by interpolation \cite[Proposition 2.13]{ADL}. Let $(X,cD)$ be its K-polystable degeneration. By Theorem \ref{thm:firstwall}, we know that $X\cong\bP^1\times\bP^1$. Since $C=4H$, we have $D=4H_0$ for some $(1,1)$-curve $H_0$. If $H_0$ is reducible, then $(X,cD)$ is isomorphic to the self-product of $(\bP^1, c[0])$. Since $(\bP^1, c[0])$ is K-unstable, we know that $(X,cD)$ is also K-unstable by \cite{Zhu19}. Thus $H_0$ must be irreducible which implies that $(\bP^1\times\bP^1, cC)\cong (X,cD)$ is K-polystable. 
Thus the proof is finished.
\end{proof}

\begin{remark}\leavevmode
\begin{enumerate}
    \item The first K-moduli wall crossing at $c_1 = \frac{1}{8}$ has the following diagram
    \[
    \oK_{\frac{1}{8}+\epsilon}\xrightarrow{\phi_1^+} \oK_{\frac{1}{8}} \xleftarrow[\cong]{\phi_1^-} \oK_{\frac{1}{8}-\epsilon}= \fM 
    \]
    where the composition $(\phi_1^{-})^{-1}\circ \phi_1^+: \oK_{\frac{1}{8}+\epsilon}\to \fM$ is 
    the Kirwan blowup of the point $[4H]$ in the GIT quotient $\fM$.  Across this wall, we replace the quadruple $(1,1)$ curve $4H$ on $\bP^1 \times \bP^1$ with GIT polystable degree $8$ curves on $\bP(1,1,2)$ which do not pass through the singular point $[0,0,1]$. This behavior is similar to \cite[Theorem 1.3]{ADL}.
   \item From Remarks \ref{rem:walls-value} and \ref{rem:walls-detail}, we will see that $c_2=\frac{1}{5}$ is the second K-moduli wall. Moreover, if a degree $8$ curve $D$ passes through the singular point of $X =\bP(1,1,2)$, then we see that for any $c < \frac{1}{5}$ the pair $(X, cD)$ is K-unstable.
\end{enumerate} \end{remark}

\subsection{Computations on CM line bundles}\label{sec:CM}
The main goals of this section are to compute the CM line bundle of the log Fano family from Section \ref{sec:LOG-VGIT}, and to show that over the complete intersection locus $U$, the CM $\bQ$-line bundle is proportional to the VGIT line bundle.

\begin{prop}\label{prop:CM-Z}
With the notation from Section \ref{sec:LOG-VGIT}, we have
\[
-f_*((-K_{\sX/\bP(E)} - c\sD)^3)=
(2-4c)^2(4c+4)\left(\eta+\frac{3c}{2c+2}\xi\right).
\]
\end{prop}

\begin{proof}
By construction we have: 
\begin{align*} \calO_{\bP^3 \times \bP(E)}(\sX) &= p_1^*\calO_{\bP^3}(2) \otimes p_2^* \pi^* \calO_{\bP^9}(1); \\ 
 \calO_{\sX}(\sD) &= p_1^* \calO_{\bP^3}(4)\vert_{\sX} \otimes p_2^*\calO_{\bP(E)}(1) \vert_{\sX}. \end{align*}
First note that $K_{\sX/\bP(E)}  = K_{\sX} - f^*K_{\bP(E)}$, and by adjunction, 
\begin{align*}
K_{\sX} &= (K_{\bP^3 \times \bP(E)} + \sX)|_{\sX} \\
& = (p_1^*\calO(-4) \otimes p_2^*\calO(K_{\bP(E)}) \otimes p_1^*\calO(2) \otimes p_2^*\pi^*\calO_{\bP^9}(1))\vert_{\sX} \\
&= \calO_{\sX}(-2) \otimes p_2^*\calO(K_{\bP(E)})\vert_{\sX} \otimes p_2^*\pi^*\calO_{\bP^9}(1)\vert_{\sX}
\end{align*}
So in particular we have  
\[K_{\sX/\bP(E)} =  \calO_{\sX}(-2) \otimes f^*\pi^*\calO_{\bP^9}(1) .\]

Since $\sD = \calO_{\sX}(4) \otimes p_2^*\calO_{\bP(E)}(1) \vert_{\sX}$, we see that 
\[ \calO_{\sX}(-K_{\sX/\bP(E)} - c\sD) = \calO_{\sX}(2-4c) \otimes f^* \pi^*\calO_{\bP^9}(-1) \otimes f^* \calO_{\bP(E)}(-c).\] 
Let $H_Y$ denote an element of the class $\calO_Y(1)$ for $Y = \sX, \bP^3, \bP(E),$ or $\bP^9$.  We compute 
\begin{align*}
    -f_*(-K_{\sX/\bP(E)} -c\sD)^3 &= -f_*(((2-4c)H_{\sX})^3 - 3((2-4c)H_{\sX})^2\cdot(cf^*H_{\bP(E)}+f^*\pi^*H_{\bP^9}) + \\
    & 3((2-4c)H_\sX \cdot (cf^*H_{\bP(E)}+f^*\pi^*H_{\bP^9})^2  ) - (cf^*H_{\bP(E)}+f^*\pi^*H_{\bP^9}) ^3) \\
    &= -f_*((2-4c)^3(\sX|_{\sX}) - 3(2-4c)^2H_\bP^3 \cdot (\sX|_{\sX})\cdot (cf^*H_{\bP(E)}+f^*\pi^*H_{\bP^9}))\\
    &= -(2-4c)^3 \pi^*H_{\bP^9} + 6(2-4c)^2(cH_{\bP(E)} + \pi^*H_{\bP^9})
\end{align*}
%
%
Thus the proof is finished since $\eta = \pi^*H_{\bP^9}$ and $\xi = H_Z$.\end{proof}

\begin{prop}\label{prop:CM-U}
Let $f_U:(\sX_U,\sD_U)\to U$ be the restriction of $f:(\sX,\sD)\to \bP(E)$ over $U\subset \bP(E)$. 
We denote the CM $\bQ$-line bundle of $f_U$ with coefficient $c$ by $\lambda_{U,c}:=\lambda_{\CM, f_U, c\sD_U}$.
Denote by $\eta_U$ and $\xi_U$ the restriction of $\eta$ and $\xi$ to $U$.  
Then for any $c\in [0,\frac{1}{2})$ we have 
\begin{equation}\label{eq:CM-U}
\lambda_{U,c}=(2-4c)^2(4c+4)\left(\eta_U+\frac{3c}{2c+2}\xi_U\right).
\end{equation}
\end{prop}

\begin{proof}
We take $l\in\bZ_{>0}$ sufficiently divisible such that $\sL:=-l(K_{\sX/\bP(E)}+c\sD)$ is a Cartier divisor on $\sX$. From the above computation, we see that $\sL\sim_{f}\cO_X(l(2-4c))$ which implies that $\sL$ is $f$-ample. 
Denote by $\sL_{U}:=\sL|_{\sX_U}$. 
Since both $\sX$ and $\bP(E)$ are smooth projective varieties, using Grothendieck-Riemann-Roch theorem, for $q\gg 1$ we have that 
\begin{align*}
\mathrm{c}_1(f_*(\sL^{\otimes q}))& =\frac{q^{3}}{6}f_*(\sL^3)-\frac{q^2}{4}f_*(K_{\sX/\bP(E)}\cdot\sL^2)+O(q),\\
\mathrm{c}_1(f_*(\sL^{\otimes q}\otimes\cO_{\sX}(-\sD)))& =\frac{q^{3}}{6}f_*(\sL^3)-\frac{q^2}{2}f_*(\sD\cdot\sL^2)-\frac{q^2}{4}f_*(K_{\sX/\bP(E)}\cdot\sL^2)+O(q).
\end{align*}
Thus $\mathrm{c}_1((f|_\sD)_*(\sL|_{\sD}^{\otimes q}))=\frac{q^2}{2}f_*(\sD\cdot\sL^2)+O(q)$. 
Since CM line bundles are functorial, by similar arguments to \cite[Proposition 2.23]{ADL} we have that 
\[
\mathrm{c}_1(\lambda_{\CM, f_U, c\sD_U, \sL_U})= -l^2 f_*((-K_{\sX/\bP(E)}-c\sD)^3)|_U.
\]
This implies \eqref{eq:CM-U} by Proposition \ref{prop:CM-Z}.
\end{proof}

\begin{prop}\label{prop:proportional} The CM $\bQ$-line bundle $\lambda_{U,c}$ and the VGIT polarization $N_t$ are proportional up to a positive constant when restricted to $U$ where $t=t(c):=\frac{3c}{2c+2}$. 
\end{prop}

\begin{proof}
By Proposition \ref{prop:CM-U}, we see that $\lambda_{U,c}$ is a positive multiple of $\eta_U + \frac{3c}{2c+2}\xi_U$. 
By Proposition \ref{prop:Linfinity}, 
\begin{align*}
    N_t |_U &= \frac{1 - 2t}{1-2\delta} p_1^*(\eta + \delta \xi) |_U + \frac{t - \delta}{2(1-2\delta)} p_2^*L_{\infty}|_U \\
    &= \frac{1 - 2t}{1-2\delta} (\eta_U + \delta \xi_U) +  \frac{t - \delta}{2(1-2\delta)} (4 \eta_U + 2 \xi_U ) \\
    &= \eta_U + t \xi_U.
\end{align*}
Hence for $t = \frac{3c}{2c+2}$, we see that $\lambda_{U,c}$ is a positive multiple of $N_t |_U$. \end{proof}

\subsection{K-moduli wall crossings and VGIT}
In this section we will prove Theorem \ref{mthm:thmintro}(2) by an inductive argument on walls.

\begin{theorem}[=Theorem \ref{mthm:thmintro}(2)]\label{thm:wallscoincide} 
Let $c \in (0, \frac{1}{2})$ be a rational number. Then there is an isomorphism between Artin stacks $\oMK_c\cong \sM(t(c))$ with $t(c)=\frac{3c}{2c+2}$. Moreover, such isomorphisms commute with wall crossing morphisms.
\end{theorem}

We first set up some notation. 
Recall that the open subset $U\subset \bP(E)$ is defined to be the locus parametrizing $(X, D)$ where $X$ is a quadric surface in $\bP^3$ and $D$ is the complete intersection of $X$ with some quartic surface in $\bP^3$.
Let $U_c^{\K}$ denote the open subset of $U$ parametrizing $c$-K-semistable log Fano pairs. Let $U_c^{\GIT}:=\sP^{\rm ss}(N_{t})$ denote the VGIT semistable locus in $\sP$ with slope $t=t(c)=\frac{3c}{2c+2}$ which is also contained in $U$ by Lemma \ref{lem:GITssU}. 
We say a point $[(X,D)]\in U$ is $c$-GIT (poly/semi)stable if it is GIT (poly/semi)stable in $\sP$ with slope $t(c)$.
By Theorem \ref{thm:generalwall}, we know that there are finitely many walls in $(0,\frac{1}{2})$ for K-moduli stacks $\oMK_c$. Denote the sequence of VGIT walls and K-moduli walls by  
\[
0=w_0<w_1<w_2<\cdots<w_{\ell}=\frac{1}{2},
\]
i.e. either $c=w_i$ is a wall for K-moduli stacks $\oMK_c$, or $t=t(w_i)$ is a wall for VGIT moduli stacks $\sM(t)$.

The following proposition allows us to replace K-moduli stacks $\oMK_c$ by a quotient stack of $U_c^{\K}$. An essential ingredient is Theorem \ref{thm:surfaces}.

\begin{prop}\label{prop:K-stackinU}
There is an isomorphism of stacks $[U_c^{\K}/\PGL(4)]\xrightarrow{\cong} \oMK_c$. Moreover, we have open immersions $
U_{c-\epsilon}^{\K}\hookrightarrow U_{c}^{\K}\hookleftarrow U_{c+\epsilon}^{\K}$
which descends (via the above isomorphisms) to wall-crossing morphisms $\oMK_{c-\epsilon}\hookrightarrow \oMK_{c}\hookleftarrow\oMK_{c+\epsilon}$.
\end{prop}

\begin{proof}
Since $U_c^{\K}$ parametrizes $c$-K-semistable log Fano pairs, by universality of K-moduli stacks we know that there exists a morphism $\psi: [U_c^{\K}/\PGL(4)]\to  \oMK_c$. In order to show $\psi$ is an isomorphism, we will construct the inverse morphism $\psi^{-1}:\oMK_c\to [U_c^{\K}/\PGL(4)]$. We follow notation from Theorem \ref{thm:modnormal}. Let $T\subset Z_{c}^{\red}$ be the connected component where a general point parametrizes $\bP^1\times\bP^1$. By Definition \ref{defn:modulispace} we know that $\oMK_c\cong [T/\PGL(N_m+1)]$. Let $T'=\pr_1(T)\subset \Hilb_{\chi}(\bP^{N_m})$. By Theorems \ref{thm:modnormal} and \ref{thm:surfaces} we know that $T'$ is smooth and contains a (possibly empty) smooth divisor $H'$ parametrizing $\bP(1,1,2)$. Moreover, both $T'\setminus H'$ and $H'$ are $\PGL(N_m+1)$-orbits in $\Hilb_{\chi}(\bP^{N_m})$.  

In order to construct $\psi^{-1}$, we will first construct a $\PGL(4)$-torsor $\cP'/T'$. The argument here is similar to \cite[Proof of Theorem 5.15]{ADL}. Let $\pi:(\cX,\cD)\to T$ and $\pi':\cX'\to T'$ be the universal families. Since $\pi'$ is an isotrivial $\bP^1\times\bP^1$-fibration over $T'\setminus H'$, there exists a flat quasi-finite morphism $\widetilde{T}\to T'$ from a smooth variety $\widetilde{T}$ that is \'etale away from $H'$ whose image intersects $H'$ (unless $H'$ is empty). From the fact that $T'\setminus H'$ and $H'$ are $\PGL(N_m+1)$-orbits, we know that there exists  $T_i'=g_i\cdot \widetilde{T}$ where $g_i\in\PGL(N_m+1)$ such that $\sqcup_i T_i'\to T$ is a fppf covering. Moreover, we may assume that $\pi'\times_{T'} (T_i'\setminus H_i'):\cX'_{T_i'\setminus H_i'}\to T_i'\setminus H_i'$ is a trivial $\bP^1\times\bP^1$-bundle for each $i$ where $H_i'=H'\times_{T'} T_i'$.  Let $\cL_i'$ be the Weil divisorial sheaf on $\cX'_{T_i'}$ as the Zariski closure of $\cO(1,1)$ on $\cX'_{T_i'\setminus H_i'}$. After replacing $T_i'$ by its Zariski covering, we may assume that $\cL_i'^{[-2]}\cong \omega_{\cX'_{T_i'}/T_i'}$. By Kawamata-Viehweg vanishing, we know that $(\pi'_{T_i'})_*\cL_i'$ is a rank $4$ vector bundle over $T_i'$. Let $\cP_i'/T_i'$ be the $\PGL(4)$-torsor induced by projectivized basis of $(\pi'_{T_i'})_*\cL_i'$. Since the cocycle condition of $\{(\pi'_{T_i'})_*\cL_i'/T_i\}_i$ is off by $\pm 1$, we know that $\{\cP_i'/T_i'\}$ is a fppf descent datum which descends to a $\PGL(4)$-torsor $\cP'/T'$ by \cite[Tag 04U1]{stacksproject}. It is clear that $\cP'/T'$ is $\PGL(N_m+1)$-equivariant. Denote by $\cP:=\cP'\times_{T'} T$. Hence the morphism $\cP\to U_c^{\K}$ given by $(t,[s_0,s_1,s_2,s_3])\mapsto [s_0,s_1,s_2,s_3](\cX_t,\cD_t)$ induces $\psi^{-1}:\oMK_c\to [U_c^{\K}/\PGL(4)]$. The proof is finished. 
\end{proof}

In order to prove Theorem \ref{thm:wallscoincide}, we run an inductive argument on the walls $w_i$. The following proposition is an initial step for induction.

\begin{prop}\label{prop:induction0}
For any $c\in (0, w_1)$, we have $U_c^{\K}=U_c^{\GIT}$.
\end{prop}

\begin{proof}
Since both $U_c^{\K}$ and $U_c^{\GIT}$ are independent of the choice of $c\in (0, w_1)$, it suffices to show that they are equal for $0< c\ll 1$.
By Theorem \ref{thm:LOmain}(2), we know that $[(X,D)]\in U_c^{\GIT}$ if and only if $X\cong \bP^1\times\bP^1$ and $D$ is a GIT semistable $(4,4)$-curve. 
By Theorem \ref{thm:firstwall} and Proposition \ref{prop:K-stackinU}, we know that $U_c^{\K}$ consists of exactly the same points as $U_c^{\GIT}$. Hence the proof is finished.
\end{proof}

Next, we divide each induction step into two statements as Propositions \ref{prop:induction1} and \ref{prop:induction2}.

\begin{prop}\label{prop:induction1}
Assume that for any $c\in (0, w_i)$ we have $U_c^{\K}=U_c^{\GIT}$. Then $U_{w_i}^{\K}=U_{w_i}^{\GIT}$.
\end{prop}

\begin{proof}
For simplicity, denote by $w:=w_i$. We first show that $U_w^{\K}\subset U_w^{\GIT}$. Let $[(X,D)]$ be a point in $U_w^{\K}$. By Proposition \ref{prop:K-stackinU}, we know that $[U_w^{\K}/\PGL(4)]\cong \oMK_w$. 
By Theorem \ref{thm:generalwall}, the K-moduli wall crossing morphism $\oK_{w-\epsilon}\to \oK_w$ is surjective which is induced by the open immersion $U_{w-\epsilon}^{\K}\hookrightarrow U_w^{\K}$. Hence there exists a $w$-K-polystable point $[(X_0, D_0)]\in U_w^{\K}$, a $(w-\epsilon)$-K-semistable point $[(X',D')]\in U_{w-\epsilon}^{\K}$, and two $1$-PS's $\sigma$ and $\sigma'$ of $\SL(4)$, such that 
\begin{equation}\label{eq:induction1}
 \lim_{t\to 0 }\sigma(t)\cdot [(X,D)]= [(X_0, D_0)],\qquad
 \lim_{t\to 0 }\sigma'(t)\cdot [(X',D')]= [(X_0, D_0)].
\end{equation}
In other words $(X_0,D_0)$ is the $w$-K-polystable degeneration of $(X,D)$, while the existence of $(X',D')$ follows from surjectivity of $\oK_{w-\epsilon}\to \oK_w$.
Denote the above two special test configurations by $(\cX, w\cD)$ and $(\cX',w\cD')$ respectively. Since $(X_0, wD_0)$ is K-polystable, we know that $\Fut(\cX',w\cD')=0$. Since the generalized Futaki invariant is proportional to the GIT weight of the CM $\bQ$-line bundle $\lambda_{U,w}$ which is again proportional to $N_t(w)|_U$ by Proposition \ref{prop:proportional}, we have that the GIT weight $\mu^{N_{t(w)}}([(X',D')], \sigma')=0$. By assumption, we have $[(X',D')]\in U_{w-\epsilon}^{\K}= U_{w-\epsilon}^{\GIT}\subset U_{w}^{\GIT}$. Hence Lemma \ref{lem:zerofut}(1) implies that $[(X_0,D_0)]\in U_{w}^{\GIT}$ which implies $[(X,D)]\in U_w^{\GIT}$ by openness of the GIT semistable locus. Thus we have shown that $U_w^{\K}\subset U_w^{\GIT}$.

Next we show the reverse containment $U_w^{\GIT}\subset U_w^{\K}$. Let $[(X,D)]$ be a point in $U_w^{\GIT}$. By almost the same argument as the previous paragraph except replacing K-stability with GIT stability, we can find $[(X_0,D_0)]\in U_w^{\GIT}$, $[(X',D')]\in U_{w-\epsilon}^{\GIT}$, and two $1$-PS's $\sigma,\sigma'$ of $\SL(4)$ such that \eqref{eq:induction1} holds, and 
\[
\mu^{N_{t(w)}}([(X,D)], \sigma)=\mu^{N_{t(w)}}([(X',D')], \sigma')=0.
\]
Note that the surjectivity of wall-crossing morphisms in VGIT follows from \cite{LO} (see Theorem \ref{thm:LOwallcrossings}).
By assumption we have $[(X',D')]\in U_{w-\epsilon}^{\GIT}=U_{w-\epsilon}^{\K}\subset U_w^{\K}$.
Again using Proposition \ref{prop:proportional} we get $\Fut(\cX',w\cD';\cL)=0$ where $(\cX',w\cD';\cL)$ is the test configuration of $(X',wD',\cO_{X'}(1))$ induced by $\sigma'$. Since $(X',wD')$ is K-semistable, by \cite[Section 8.2]{LX14} we know that $\cX'$ is regular in codimension $1$. Since $\cX_0'=X_0$ is Cohen-Macaulay, we know that $\cX'$ is $S_2$ which implies that $\cX'$ is normal. Hence Lemma \ref{lem:zerofut}(2) implies that $(X_0, wD_0)$ is K-semistable, and so is $(X,wD)$ by the openness of K-semistability \cite{BLX19, Xu19}. The proof is finished.
\end{proof}

\begin{prop}\label{prop:induction2}
Assume that for any $c\in (0, w_i]$ we have $U_c^{\K}=U_c^{\GIT}$. Then $U_{c'}^{\K}=U_{c'}^{\GIT}$ for any $c'\in (w_i, w_{i+1})$.
\end{prop}

\begin{proof}
For simplicity, denote by $w:=w_i$. 
Since the K-semistable locus $U_{c'}^{\K}$ and the GIT semistable locus $U_{c'}^{\GIT}$ are independent of the choice of  $c'\in (w_i,w_{i+1})$, it suffices to show that $U_{w+\epsilon}^{\K} = U_{w + \epsilon}^{\GIT}$. 
We first show $U_{w+\epsilon}^{\K} \subset U_{w + \epsilon}^{\GIT}$. Assume to the contrary that $[(X,D)]\in U_{w+\epsilon}^{\K}\setminus U_{w+\epsilon}^{\GIT}$. 
We note that by Proposition \ref{prop:K-stackinU} and Lemma \ref{lem:VGITbasics} there are open immersions $U_{w+\epsilon}^\K \hookrightarrow U_w^{\K}$ and $U_{w+\epsilon}^{\GIT} \hookrightarrow U_w^{\GIT}$.
By assumption we have $[(X,D)]\in U_{w+\epsilon}^\K\subset U_w^{\K}=U_w^{\GIT}$, hence $[(X,D)]$ is $w$-GIT semistable but $(w+\epsilon)$-GIT unstable. Thus by Lemma \ref{lem:VGITbasics}
there exists a 1-PS $\sigma: \bG_m \to \SL(4)$ such that 
\begin{equation}\label{eq:induction2}
\mu^{N_{t(w)}}([(X,D)], \sigma)=0, \qquad \mu^{N_{t(w+\epsilon)}}([(X,D)], \sigma)<0.
\end{equation}
Denote by $\zeta_0:=\lim_{t\to 0}\sigma(t)\cdot [(X,D)] \in \sP$. Since $[(X,D)]$ is $w$-GIT semistable, by Lemma \ref{lem:zerofut}(1) and \eqref{eq:induction2} we know that $\zeta_0$ is also $w$-GIT semistable, in particular $\zeta_0=[(X_0,D_0)]\in U$. Denote by  $(\cX,w\cD;\cL)/\bA^1$ the test configuration of $(X,wD;\cO_X(1))$ induced by $\sigma$. Hence by \eqref{eq:induction2} and Proposition \ref{prop:proportional}, we have $\Fut(\cX,(w+\epsilon)\cD)<0$. This implies that $(X,(w+\epsilon)D)$ is K-unstable which contradicts the assumption that $[(X,D)]\in U_{w+\epsilon}^{\K}$. Thus we conclude that $U_{w+\epsilon}^{\K}\subset U_{w+\epsilon}^{\GIT}$.

Next, if $[(X,D)] \in U_{w+\epsilon}^\K$ is $(w+\epsilon)$-K-polystable, then we claim that $[(X,D)]$ is $(w+\epsilon)$-GIT polystable. We have already shown that $[(X,D)]$ is $(w+\epsilon)$-GIT semistable. Let us take a  1-PS $\sigma'$ of $\SL(4)$ degenerating $[(X,D)]$ to a $(w+\epsilon)$-GIT polystable point $[(X',D')]$. Hence we have $\mu^{N_{t(w+\epsilon)}}([(X,D)],\sigma')=0$. 
By Proposition \ref{prop:proportional}, we have $\Fut(\cX',(w+\epsilon)\cD';\cL')=0$ where $(\cX',(w+\epsilon)\cD';\cL')$ is the test configuration of $(X,(w+\epsilon)D;\cO_X(1))$ induced by $\sigma'$. Since $[(X',D')]\in U_{w+\epsilon}^{\GIT}\subset U_w^{\GIT}=U_w^{\K}$ by assumption, we know that $(X',wD')$ is K-semistable hence klt. Thus $(\cX',(w+\epsilon)\cD')$ is a special test configuration with vanishing generalized Futaki invariant. Since $(X,(w+\epsilon)D$ is K-polystable, we know that $(X,D)\cong (X',D')$ which implies that $[(X,D)]$ and $[(X',D')]$ belong to the same $\SL(4)$-orbit in $U$. Hence $[(X,D)]$ is $(w+\epsilon)$-GIT polystable.

Finally we show that $U_{w+\epsilon}^{\K}= U_{w+\epsilon}^{\GIT}$. Consider the following commutative diagram
\begin{center}
\begin{tikzcd}
U_{w+\epsilon}^{\K} \arrow[d, hook, "f"] \arrow [r] & {[U_{w+\epsilon}^{\K}/\PGL(4)]} \arrow[d, hook, "g"]\arrow [r] & U_{w+\epsilon}^{\K}\sslash \PGL(4) \arrow[d, "h"]\\
U_{w+\epsilon}^{\GIT}  \arrow [r] & {[U_{w+\epsilon}^{\GIT}/\PGL(4)]} \arrow[r]& U_{w+\epsilon}^{\GIT}\sslash \PGL(4)
\end{tikzcd} 
\end{center}
Since $f$ is an open immersion between smooth varieties, its descent $g$ is separated and representable. By Lemma \ref{prop:K-stackinU} we know $[U_{w+\epsilon}^{\K}/\PGL(4)]\cong \oMK_{w+\epsilon}$, hence $g$ maps closed points to closed points as shown in the previous paragraph, and $h$ is quasi-finite. Since the GIT quotients on the third column are isomorphic to the K-moduli space $\oK_{w+\epsilon}$ and the VGIT moduli space $\fM(t(w+\epsilon))$ respectively, they are both proper. Thus $h$ is a finite morphism. Then we apply \cite[Proposition 6.4]{alper} to conclude that $g$ is a finite morphism as well. In particular, this implies that $f$ is finite hence surjective. The proof is finished. 
\end{proof}

\begin{proof} [Proof of Theorem \ref{thm:wallscoincide}]
By Propositions \ref{prop:induction0}, \ref{prop:induction1}, and \ref{prop:induction2} on induction of the walls $\{w_i\}_{i=0}^{\ell}$, we conclude that $U_c^{\K}=U_c^{\GIT}$ for any $c\in (0,\frac{1}{2})$. Hence the theorem follows from Proposition \ref{prop:K-stackinU} and the definition $\sM(t(c))=[U_c^{\GIT}/\PGL(4)]$.
\end{proof}

\begin{proof}[Proof of Theorem \ref{mthm:thmintro}]
Part (1) follows from Theorem \ref{thm:firstwall}. Part (2) is precisely Theorem \ref{thm:wallscoincide}.
\end{proof}

\begin{proof}[Proof of Theorem \ref{mthm:spaceiso}]
The first isomorphism follows from Theorem \ref{mthm:thmintro}. The second isomorphism follows from Theorem \ref{thm:LOwallcrossings}. For the proportionality statements, the first one between CM $\bQ$-line bundle and VGIT polarization follows from Proposition \ref{prop:proportional}, while the second one between VGIT polarization and push forward of $\lambda+\beta\Delta$ follows from \cite[Proposition 7.6]{LO}.
\end{proof}

\begin{proof}[Proof of Theorem \ref{mthm:slcK3}]
Since there are finitely many K-moduli (resp. GIT) walls for $c\in (0,\frac{1}{2})$ (resp. $t\in (0,\frac{1}{2})$),  we may assume that $\epsilon$ and $\epsilon'$ satisfy the relation $\epsilon=\frac{3\epsilon'}{2\epsilon'+2}$, i.e.  $\frac{1}{2}-\epsilon' = t(\frac{1}{2}-\epsilon)$.
By Theorem \ref{mthm:thmintro}, we have $\fM(\frac{1}{2}-\epsilon')\cong \oK_{\frac{1}{2}-\epsilon}$. The isomorphism $\fM(\frac{1}{2}-\epsilon')\cong\widehat{\sF}$ follows from \cite[Theorem 1.1]{LO}. 

For part (1), 
from the above isomorphisms we know that $\fM(\frac{1}{2}-\epsilon')$ parametrizes K-polystable klt log Fano pairs $(X,(\frac{1}{2}-\epsilon')D)$. By ACC of log canonical thresholds \cite{HMX14}, we know that $(X,\frac{1}{2}D)$ is log canonical. Hence taking double cover of $X$ branched along $D$ we obtain a hyperelliptic K3 surface $S$ with only slc singularities. The proof is finished.

For part (2), notice that by taking fiberwise double covers of the universal log Fano family over $\oMK_{\frac{1}{2}-\epsilon}$, we obtain a universal family of slc K3 surfaces $\cS\to \cT$ where $\cT\to \oMK_{\frac{1}{2}-\epsilon}$ is a $\bm{\mu}_2$-gerbe. In particular, the Hodge line bundle $\lambda_{\Hodge,\cT}$ of the K3 family $\cS/\cT$ is the pull-back of the Hodge line bundle $\lambda_{\Hodge, \frac{1}{2}-\epsilon}$ over $\oMK_{\frac{1}{2}-\epsilon}$. Taking good moduli spaces of $\cT\to T$ and $\oMK_{\frac{1}{2}-\epsilon}\to \oK_{\frac{1}{2}-\epsilon}$ gives an isomorphism $T\xrightarrow{\cong}\oK_{\frac{1}{2}-\epsilon}$. Since both spaces are isomorphic to $\widehat{\sF}$, we know that $\sF$ admits an open immersion into $T$ whose complement has codimension at least $2$.  In particular, we know that $\lambda_{\Hodge, T}|_{\sF}=\lambda_{\Hodge, \sF}$, and the conclusion follows from $\sF^*=\Proj R(\sF, \lambda_{\Hodge, \sF})$.
\end{proof}

\begin{rem}\label{rem:walls-value}
According to \cite{LO}, the $t$-walls for VGIT quotients $\fM(t)$ and $\beta$-walls for the Hassett-Keel-Looijenga program for $\sF(\beta)=\Proj R(\sF,\lambda+\beta\Delta)$ with $N=18$ (under the transformation rule $t=\frac{1}{4\beta+2}$) are given by 
\[
t\in \left\{\frac{1}{6}, \frac{1}{4}, \frac{3}{10}, \frac{1}{3}, \frac{5}{14}, \frac{3}{8}, \frac{2}{5}, \frac{1}{2} \right\},\qquad
\beta\in \left\{1, \frac{1}{2}, \frac{1}{3}, \frac{1}{4}, \frac{1}{5}, \frac{1}{6}, \frac{1}{8}, 0 \right\}.
\]
By the transformation rule $t=\frac{3c}{2c+2}$, we obtain the $c$-walls for K-moduli stacks $\oMK_c$ are
\[
c\in \left\{\frac{1}{8}, \frac{1}{5}, \frac{1}{4}, \frac{2}{7}, \frac{5}{16}, \frac{1}{3}, \frac{4}{11}, \frac{1}{2} \right\}.
\]
Note that $c=\frac{1}{2}$ corresponds to the log Calabi-Yau wall crossing $\oK_{\frac{1}{2}-\epsilon}\to \sF^*$, while the rest walls are in the log Fano region.
\end{rem}

\begin{rem}\label{rem:walls-detail}(cf. \cite[Section 6]{LO}) Let $i\in \{1,2,\cdots, 7\}$ be an index. For the $i$-th K-moduli wall $c_i$, we have K-moduli wall crossing morphisms
\[
\oK_{c_i-\epsilon}\xrightarrow{\phi_i^{-}}\oK_{c_i}\xleftarrow{\phi_i^{+}}\oK_{c_i+\epsilon}.
\]
Denote by $\Sigma_i^{\pm}$ the closed subset of $\oK_{c_i\pm\epsilon}$ parametrizing pairs that are  $(c_i\pm \epsilon)$-K-polystable but not $c_i$-K-polystable. As observed in \cite[Section 6]{LO}, we know that a general point $[(X,D)]$ in $\Sigma_i^{-}$ (resp.  $\Sigma_i^{+}$) parametrizes a curve $D$ on $X\cong\bP^1\times\bP^1$ (resp. $X\cong\bP(1,1,2)$).
In Table \ref{table:singularities}, we rephrase results from \cite{LO}, especially \cite[Table 2]{LO}, to describe the generic singularities (in local analytic form) presented in the curves $D$. 
Note that a general curve $D$ in $\Sigma_i^+$ is smooth when $i=1$, and singular only at the cone vertex $v=[0,0,1]$ of $\bP(1,1,2)$ when $2\leq i\leq 7$.

\begin{table}[htbp!]\renewcommand{\arraystretch}{1.5}
\caption{Singularities along the K-moduli  walls}\label{table:singularities}
\begin{tabular}{|c|c|l|l|}
\hline 
$i$ & $c_i$   & \textbf{Sing. of $D$ in $\Sigma_i^-$} & \textbf{Sing. of $D$ in $\Sigma_i^+$}\\ \hline \hline 
1 & $\frac{1}{8}$  & quadruple conic     &     $v\not\in D$                   \\ 
2 & $\frac{1}{5}$  & triple conic + transverse conic  &   $A_1$         \\ 
3 & $\frac{1}{4}$  & $J_{4,\infty}: ~x^3+x^2y^4=0$     &     $A_2$            \\ 
4 & $\frac{2}{7}$  & $J_{3,0}:~ x^3 + b_1 x^2y^3 + y^9 + b_2 xy^7=0$       &   $A_3$                    \\ 
5 & $\frac{5}{16}$ & $E_{14}:~ x^3 + y^8 + axy^6 = 0$  &     $A_4$      \\ 
6 & $\frac{1}{3}$  & $E_{13}:~ x^3 + xy^5 + ay^8 = 0$      & $A_5$              \\ 
7 & $\frac{4}{11}$ & $E_{12}:~ x^3 + y^7 + axy^5 = 0$      &  $A_7$             \\ \hline
\end{tabular}
\end{table}

\end{rem}

\section{Some results for $(d,d)$ curves}\label{sec:generaldegree}

In this section we discuss some generalizations of our results to $(d,d)$-curves on $\bP^1\times\bP^1$ including the proof of Theorem \ref{mthm:alldeg}. We  assume $d\geq 3$ throughout this section. 

\subsection{VGIT for $(2,d)$ complete intersections in $\bP^3$}
Let $\bfP_{(d,d)}:=\bP(H^0(\bP^1\times\bP^1,\cO(d,d)))$. We say a $(d,d)$-curve $C$ on $\bP^1\times\bP^1$ is \emph{GIT (poly/semi)stable} if $[C]$ is GIT (poly/semi)stable with respect to the natural $\Aut(\bP^1\times\bP^1)$-action on $(\bfP_{(d,d)},\cO(2))$. We define the GIT moduli stack $\sM_d$ and the GIT moduli space $\fM_d$ of degree $(d,d)$ curves as
\[
\sM_d:= [\bfP_{(d,d)}^{\rm ss}/\Aut(\bP^1\times\bP^1)],\qquad
\fM_d:=\bfP_{(d,d)}^{\rm ss}\sslash\Aut(\bP^1\times\bP^1).
\]

Next, we describe the VGIT of $(2,d)$ complete intersection curves in $\bP^3$ based on \cite{benoist, CMJL14, LO}. Our set-up is a direct generalization of Section \ref{sec:LOG-VGIT}. Let 
\[
\pi:\bP(E_d)\to \bP(H^0(\bP^3,\cO(2)))=\bP^9
\]
be the projective space bundle with fiber $\bP(H^0(Q,\cO_Q(d)))$ over a quadric surface $[Q]\in\bP^9$. Let $f: (\sX,\sD)\to \bP(E_d)$ the universal family of quadric surfaces with $(2,d)$ intersections over $\bP(E_d)$. Denote by $\eta:=\pi^*\cO_{\bP^9}(1)$ and $\xi:=\cO_{\bP(E_d)}(1)$.  Then we have the following result of Benoist, where a special case of $d=4$ is stated in Proposition \ref{prop:benoist2,4}.

\begin{prop}\label{prop:benoist-alldeg}\cite[Theorem 2.7]{benoist}
If $t \in \bQ$, then the $\bQ$-Cartier class $\oN_t:=\eta + t\xi$ on $\bP(E_d)$ is ample if and only if $t \in (0, \frac{1}{d-1}) \cap \bQ$. 
\end{prop}

Let $U_{(2,d)}\subset\bP(E_d)$ be the complete intersection locus as an open subset. Then we know that $\codim_{\bP(E_d)}\bP(E_d)\setminus U_{(2,d)}\geq 2$. There is a birational morphism $\chow: U_{(2,d)}\to \Chow_{(2,d)}$ as a restriction of the Hilbert-Chow morphism. Hence the graph of $\chow$ gives a locally closed embedding
\[
U_{(2,d)}\hookrightarrow \bP(E_d)\times \Chow_{(2,d)}.
\]
Denote by $\sP_d$ the closure of $U_{(2,d)}$ in $\bP(E_d)\times \Chow_{(2,d)}$. Let $p_1$ and $p_2$ be the first and second projections from $\sP_d$ to $\bP(E_d)$ and $\Chow_{(2,d)}$, respectively. The action of $\SL(4)$ on $\bP^3$ extends naturally to actions on $U_{2,d}$, $\bP(E_d)$, $\Chow_{(2,d)}$, and $\mathscr{P}_d$. Similar to Section \ref{sec:LOG-VGIT}, we will specify a family of $\SL(4)$-linearized ample $\bQ$-line bundles on $\sP_d$. 

Fix a rational number $0 < \delta < \frac{2}{3d}$.  For $t \in (\delta, \frac{2}{d}] \cap \bQ$, consider the $\mathbb{Q}$-line bundle 
\[N_t := \frac{2 - dt}{2-d\delta} p_1^*(\eta + \delta \xi) + \frac{t - \delta}{2-d\delta} p_2^*L_{\infty} ,\]
where $L_{\infty}$ is the restriction of the natural polarization of the Chow variety to $\Chow_{(2,d)}$. Since $\frac{2}{3d}<\frac{1}{d-1}$, Proposition \ref{prop:benoist-alldeg} implies that $\eta+\delta\xi$ is ample on $\bP(E_d)$. It is clear that $L_\infty$ is ample on $\Chow_{(2,d)}$. Hence $N_t$ is ample for $\delta<t<\frac{2}{d}$ and semiample for $t=\frac{2}{d}$.

\begin{definition}\label{def:VGIT-alldeg}
Let $\delta \in \mathbb{Q}$ satisfy $0 < \delta < \frac{2}{3d}$. For each $t \in (\delta, \frac{2}{d}) \cap \bQ$, we define the VGIT quotient stack $\sM_d(t)$  and the VGIT quotient space $\fM_d(t)$ of slope $t$ to be
\[ \sM_d(t) := [\sP_d^{\rm ss}(N_t)/\PGL(4)], \quad \fM_d(t):=\sP_d\sslash_{N_t} \SL(4).\]
\end{definition}

The above definition a priori depends on the choice of $\delta\in (0,\frac{2}{3d})$. Nevertheless, similar to \cite{LO} we will show in Theorem \ref{thm:LOmain-alldeg}(1) that both $\sM_d(t)$ and $\fM_d(t)$ do not depend on the choice of $\delta$, hence are well-defined for all $t\in (0,\frac{2}{d})$. Before stating the main VGIT result Theorem \ref{thm:LOmain-alldeg}, we need some preparation. 


\begin{lem}\label{lem:proportional-alldeg}
With notation as above, we have $N_t|_{U_{(2,d)}}= \oN_t|_{U_{(2,d)}}$ for any $t\in (\delta,\frac{2}{d}]\cap \bQ$. 
\end{lem}

\begin{proof}
Denote by $\oL_\infty$ the unique extension of $L_\infty|_{U_{(2,d)}}$ to $\bP(E_d)$. By the same argument as \cite[Proposition 5.4]{LO}, we get that $\oL_\infty=d\eta+2\xi$. Hence we have
\begin{align*}
N_t|_{U_{(2,d)}}& =\frac{2 - dt}{2-d\delta} (\eta + \delta \xi)|_{U_{(2,d)}} + \frac{t - \delta}{2-d\delta} \oL_{\infty}|_{U_{(2,d)}}\\
& = \frac{2 - dt}{2-d\delta} (\eta + \delta \xi)|_{U_{(2,d)}} + \frac{t - \delta}{2-d\delta} (d\eta+2\xi)|_{U_{(2,d)}} = (\eta+t\xi)|_{U_{(2,d)}}.
\end{align*}
The proof is finished.
\end{proof}

The following  lemma is very useful (see \cite[Propositions 4.6 and 6.2]{CMJL14} and Lemma \ref{lem:GITssU} for $d=3,4$).

\begin{lem}\label{lem:GITssU-alldeg}
For each $t\in (\delta,\frac{2}{d})\cap\bQ$ (resp. $t\in (0,\frac{1}{d-1})\cap\bQ)$, the VGIT semistable locus $\sP_d^{\rm ss}(N_t)$ (resp. $\bP(E_d)^{\rm ss}(\oN_t)$) of slope $t$ is a Zariski open subset of $U_{(2,d)}$. 
\end{lem}

\begin{proof}
We first consider the VGIT semistable locus of $\bP(E_d)$.
Let $([Q], [s])$ be a point in $\bP(E_d)\setminus U_{(2,d)}$ where $Q=(q=0)$ is a non-normal quadric surface in $\bP^3$ and $0\neq s\in H^0(Q,\cO_Q(d))$. Let $g\in H^0(\bP^3, \cO_{\bP^3}(d))$ be a lifting of $s$. We choose suitable projective coordinates $[x_0,x_1,x_2,x_3]$ of $\bP^3$ such that one of the following holds:
\begin{enumerate}[label=(\alph*)]
    \item $q=x_0 x_1$, and $g=x_0 h$ where $h\in \bC[x_0,\cdots, x_3]_{d-1}$, and $x_1\nmid h$.
    \item $q=x_0^2$, and $g=x_0 h$ where  $h\in \bC[x_0,\cdots, x_3]_{d-1}$, and $x_0\nmid h$.
\end{enumerate}
Let $\sigma$ be the $1$-PS in $\SL(4)$ of weights $(-3,1,1,1)$ with respect to the chosen coordinates. By \cite[Proposition 2.15]{benoist}, for any $t\in (0,\frac{2}{d}]$ we have
\[
\mu^{\oN_t}(([Q], [s]), \sigma)\leq \mu(q,\sigma)+t\mu(g,\sigma)\leq -2+t(d-4)<0.
\]
Hence  $([Q], [s])$ is VGIT unstable of slope $t$ by the Hilbert-Mumford numerical criterion. 

Next, we consider the VGIT semistable locus of $\sP_d$. It is clear that any point $z$ in $\sP_d\setminus U_{(2,d)}$ has the form $z=(([Q],[s]), \chow(\sC))$ where $([Q],[s])\in \bP(E_d)\setminus U_{(2,d)}$, $\sC\in \Hilb_{(2,d)}\setminus U_{(2,d)}$, and $\chow: \Hilb_{(2,d)}\to \Chow_{(2,d)}$ is the Hilbert-Chow morphism. We choose $[x_0,\cdots,x_3]$ and $\sigma$ as above. Then
\[
\mu^{N_t}(z, \sigma)=\frac{2-dt}{2-d\delta} \mu^{\oN_{\delta}}(([Q],[s]),\sigma)+\frac{t-\delta}{2-d\delta}\mu^{L_\infty}(\chow(\sC),\sigma).
\]
From the above argument we get $\mu^{\oN_{\delta}}(([Q],[s]),\sigma)<0$. By \cite[Propostion 5.8]{LO} we know that $\mu^{L_\infty}(\chow(\sC),\sigma)<0$. Hence $\mu^{N_t}(z,\sigma)<0$ for any $t\in (\delta,\frac{2}{d})\cap\bQ$ and the proof is finished.
\end{proof}

Indeed, we have a stronger result on VGIT semistable loci (see \cite[Lemma 6.8]{LO} for $d=4$). 

\begin{lem}\label{lem:GITssnormal}
For each $t\in (\delta,\frac{2}{d})\cap\bQ$ (resp. $t\in (0,\frac{1}{d-1})\cap\bQ)$, any VGIT semistable point in  $\sP_d^{\rm ss}(N_t)$ (resp. $\bP(E_d)^{\rm ss}(\oN_t)$) of slope $t$ has the form $([Q],[s])$ where $\rank(Q)\geq 3$.
\end{lem}

\begin{proof}
Let $z=([Q],[s])$ be a point in $U_{(2,d)}$ where $\rank(Q)\leq 2$. Hence by Lemma \ref{lem:GITssU-alldeg} it suffices to show instability of $z$ in $\bP(E_d)$ and $\sP_d$ respectively. We will assume $t\in (0,\frac{2}{d})\cap \bQ$ throughout the proof. Choose a projective coordinate $[x_0,\cdots,x_3]$ such that $Q=(q=0)$ is defined by $q=x_0^2$ or $x_0x_1$. Let $g\in H^0(\bP^3, \cO_{\bP^3}(d))$ be a lifting of $s$. Let $\sigma$ be the $1$-PS in $\SL(4)$ of weights $(-1,-1,1,1)$ with respect to the chosen coordinates. Then by \cite[Proposition 2.15]{benoist}
\[
\mu^{\oN_t}(z, \sigma)\leq \mu(q,\sigma)+t\mu(g,\sigma)\leq -2+td<0.
\]
Hence $z$ is $\oN_t$-unstable in $\bP(E_d)$. It is clear that $\lim_{r\to 0}\lambda(r)\cdot ([Q],[s])=([Q], [g(0,0,x_2,x_3)])$ in $\bP(E_d)$. Hence for general $s$ we see that $\lim_{r\to 0}\lambda(r)\cdot ([Q],[s])$ belongs to $U_{(2,d)}$. In particular, Lemma \ref{lem:proportional-alldeg} implies that $\mu^{N_t}(z,\sigma)=\mu^{\oN_t}(z,\sigma)<0$, so $z$ is $N_t$-unstable in $\sP_d$ when $s$ is general. Since the GIT unstable locus is closed, we conclude that $z$ is  $N_t$-unstable for any choice of $s$. 
\end{proof}

The following theorem is a generalization of \cite[Theorem 5.6]{LO}.

\begin{theorem}\label{thm:LOmain-alldeg}
Let $\delta$ be as above. The following hold:
\begin{enumerate}
    \item The VGIT semistable locus $\sP_d^{\rm ss}(N_t)$ is independent of the choice of $\delta$.
    \item For $t \in (\delta, \frac{1}{d-1})$, we have $\sM_d(t)\cong [\bP(E_d)^{\rm ss}(\oN_t)/\PGL(4)]$  and  $\fM_d(t) \cong \bP(E_d) \sslash_{\oN_t} \SL(4)$. 
    \item For $t \in (\delta, \frac{2}{3d})$, we have $\sM_d(t)\cong \sM_d$ and $\fM_d(t) \cong \fM_d$. 
\end{enumerate}
\end{theorem}

\begin{proof}
(1) Let $\delta$ and $\delta'$ be two rational numbers in $(0, \frac{2}{3d})$. Denote by $G:=\SL(4)$. Denote the corresponding polarization on $\sP_d$ by $N_t$ and $N_t'$. Since both GIT semistable loci $\sP_d$ with respect to $N_t$ and $N_t'$ are contained in $U_{(2,d)}$ where their restrictions are the same by Lemmas \ref{lem:proportional-alldeg} and \ref{lem:GITssU-alldeg}, \cite[Lemma 4.17]{CMJL14} implies that for $m\in\bN$ sufficiently divisible we have
\[
H^0(\sP_d, N_t^{\otimes m})^G\xrightarrow{\cong } H^0(U_{(2,d)}, N_t|_{U_{(2,d)}}^{\otimes m})^G=H^0(U_{(2,d)}, N_t'|_{U_{(2,d)}}^{\otimes m})^G\xleftarrow[]{\cong} H^0(\sP_d, N_t'^{\otimes m})^G.
\]
Since both $\sP_d^{\rm ss}(N_t)$ and $\sP_d^{\rm ss}(N_t')$ are the union of non-vanishing loci of $G$-invariant sections in the first and last terms of the above diagram, we know that they are equal. Hence $\sP_d^{\rm ss}(N_t)$ is independent of the choice of $\delta$. 

(2) The proof is similar to (1) using Lemmas \ref{lem:proportional-alldeg}, \ref{lem:GITssU-alldeg}, and \cite[Lemma 4.17]{CMJL14}.

(3) By (2) it suffices to show that $[\bP(E_d)^{\rm ss}(\oN_t)/\PGL(4)]\cong \sM_d$ for $t\in (0, \frac{2}{3d})$. By Lemma \ref{lem:GITssnormal}, we know that any GIT semistable point $z\in \bP(E_d)$ with respect to $\oN_t$ has the form $z=([Q],[s])$ where $\rank(Q)\geq 3$. We will show that under the assumption $t<\frac{2}{3d}$ the quadric surface $Q$ must be smooth. Assume to the contrary that $Q=(q=0)$ is singular. Then we may choose a projective coordinate $[x_0,\cdots,x_3]$ of $\bP^3$ such that $q\in \bC[x_1,x_2,x_3]_2$. Let $\sigma$ be the $1$-PS in $\SL(4)$ with weights $(3,-1,-1,-1)$. Let $g\in H^0(\bP^3,\cO_{\bP^3}(d))$ be a lifting of $s$. Then by \cite[Proposition 2.15]{benoist} we have
\[
\mu^{\oN_t}(z, \sigma)\leq \mu(q, \sigma)+t\mu(g, \sigma)\leq -2+t\cdot 3d<0.
\]
Hence $z$ is $\oN_t$-unstable on $\bP(E_d)$. Since $\sigma$ fixes $Q$, we know that $\lim_{r\to 0}\sigma(r)\cdot z$ belongs to $U_{(2,d)}$. Hence $\mu^{N_t}(z, \sigma)=\mu^{\oN_t}(z, \sigma)<0$ by Lemma \ref{lem:proportional-alldeg} which implies that $z$ is $N_t$-unstable on $\sP_d$. The rest of the proof is similar to \cite[Lemma 4.18]{CMJL14}.
\end{proof}

\begin{rem}
When $t=\frac{2}{d}$, we can define the VGIT quotient stack and space by 
\[
\sM_d(\tfrac{2}{d}):=[\Chow_{(2,d)}^{\rm ss}/\PGL(4)],\qquad \fM_d(\tfrac{2}{d}):=\Chow_{(2,d)}\sslash \SL(4).
\]
As in \cite{LO}, one can show that there are natural wall crossing morphisms  $\sM_d(\frac{2}{d}-\epsilon)\to \sM_d(\frac{2}{d})$ and $\fM_d(\frac{2}{d}-\epsilon)\to \fM_d(\frac{2}{d})$ for $0<\epsilon\ll 1$. We omit further discussion on the Chow quotient since it is not directly related to our K-moduli spaces when $d\neq 4$ (see e.g. Remark \ref{rem:OSS}).
\end{rem}

\subsection{Proofs}
In this section we prove Theorem \ref{mthm:alldeg}.
We first prove part (1) of Theorem \ref{mthm:alldeg}.


\begin{proof}[Proof of Theorem \ref{mthm:alldeg}(1)]
The proof is similar to Theorem \ref{thm:firstwall}. Consider the universal family $\pi_d: (\bP^1 \times \bP^1 \times \mathbf{P}_{(d,d)}, c\calC) \to \mathbf{P}_{(d,d)}$ over the parameter space of $(d,d)$-curves on $\bP^1\times\bP^1$. It is clear that $\calC \in |\cO(d,d,1)|$.  Hence by Proposition \ref{prop:logCM2} we know that the CM $\bQ$-line bundle $\lambda_{\CM, \pi_d, c\calC}$ is equal to  $\cO_{\bfP_{(d,d)}}(3(2-dc)^2 c)$ which is ample for $c\in (0,\frac{2}{d})$. Hence K-(poly/semi)stability of $(\bP^1\times\bP^1, cC)$ implies GIT (poly/semi)stability of $C$. For the other direction, let $(X,cD)$ be a K-semistable pair parametrized by $\oMK_{d,c}$ with $c\in (0, \frac{1}{2d})$. By \cite{LL16}, for any point $x\in X$ we have
\[
\hvol(x,X)\geq \hvol(x,X,cD)\geq \frac{4}{9}(-K_X-cD)^2=\frac{32}{9}(1-dc)^2>2.
\]
This implies that any $x\in X$ is smooth, hence $X\cong\bP^1\times\bP^1$. The rest of the proof is exactly the same as Theorem \ref{thm:firstwall}.
\end{proof}

\begin{rem}
Similar to Proposition \ref{prop:firstwallreplace}, we have that $c_1=\frac{1}{2d}$ is the first K-moduli wall for $(d,d)$-curves on $\bP^1\times\bP^1$ which replaces $(\bP^1\times\bP^1, dH)$ by $(\bP(1,1,2), D)$ where $H$ is a smooth $(1,1)$-curve.
\end{rem}



Next, we prove part (2) of Theorem \ref{mthm:alldeg}. Before starting the proof, we need some preparation on CM line bundles as a generalization of Propositions \ref{prop:CM-U} and \ref{prop:proportional}. 

\begin{prop}\label{prop:CM-U-alldeg}
For simplicity, denote by $U:=U_{(2,d)}$.
Let $f_U:(\sX_U,\sD_U)\to U$ be the restriction of $f:(\sX,\sD)\to \bP(E_d)$ over $U\subset \bP(E_d)$. 
We denote the CM $\bQ$-line bundle of $f_U$ with coefficient $c$ by $\lambda_{U,c}:=\lambda_{\CM, f_U, c\sD_U}$.
Then $\lambda_{U,c}$ and $N_t|_U$ are proportional up to a positive constant where $t=t(c):=\frac{6c}{dc+4}$ and $c\in (0,\frac{2}{d})$. 
\end{prop}

\begin{proof}
By the same computations as Section \ref{sec:CM}, we get $\lambda_{U,c}=(2-dc)^2(dc+4)(\eta+\frac{6c}{dc+4}\xi)|_U$.
\end{proof}

\begin{proof}[Proof of Theorem \ref{mthm:alldeg}(2)] 
We first fix some notation. Let $U_c^{\K}$ be the open subset of $U=U_{(2,d)}$ parametrizing $c$-K-semistable log Fano pairs. Let $U_c^{\GIT}:=\sP_d^{\rm ss}(N_t)$ be the open subset of $U$ parametrizing VGIT semistable points of slope $t=t(c)=\frac{6c}{dc+4}$. Similar to Proposition \ref{prop:K-stackinU}, by Theorem \ref{thm:surfacesalld} we know that $[U_c^{\K}/\PGL(4)]\cong \oMK_{d,c}$ as long as $c\in (0, \frac{4-\sqrt{2}}{2d})$. Hence it suffices to show $U_c^{\K}=U_c^{\GIT}$ for $c\in (0,\frac{4-\sqrt{2}}{2d})$. 

We follow the strategy in the proof of Theorem \ref{thm:wallscoincide}, that is, by induction on the walls for K-moduli and VGIT. It suffices to generalize Propositions \ref{prop:induction0}, \ref{prop:induction1}, and \ref{prop:induction2} to $(2,d)$ complete intersections under the assumption $c<\frac{4-\sqrt{2}}{2d}$. The generalization of Proposition \ref{prop:induction0} follows from Theorems \ref{mthm:alldeg}(1) and \ref{thm:LOmain-alldeg}(3). For Propositions \ref{prop:induction1} and \ref{prop:induction2}, we can generalize them using $[U_c^{\K}/\PGL(4)]\cong \oMK_{d,c}$, Proposition \ref{prop:CM-U-alldeg}, and Theorem \ref{thm:generalwall}.
\end{proof}


\begin{rem}\label{rem:OSS}
If $d\neq 4$ then the isomorphism $\oK_{d,c}\cong \fM_d(t)$ can fail for $c>\frac{4-\sqrt{2}}{2d}$. For instance, it was observed in \cite[Example 5.8]{OSS16} that $\bP(1,2,9)$ appears in the K-moduli space $\oK_{3,\frac{1}{2}}$. We will further investigate the case $d=3$ in a forthcoming work. It would also be interesting to consider more general divisors as well as other del Pezzo surfaces. 
\end{rem}

\begin{rem}
In the forthcoming work \cite{ADL21}, we give a complete description of wall-crossing for K-moduli compactifications of $(\bP^3, cS)$ where $S\subset \bP^3$ is a smooth degree $4$ K3 surface. As an application, we prove Laza-O'Grady's conjecture \cite{LO16, LO18b} on birational models of moduli of degree $4$ K3 surfaces. An essential ingredient is Theorem \ref{mthm:thmintro} which fully describes the wall-crossing behavior for K-moduli spaces of hyperelliptic quartic K3 surfaces.
\end{rem}

\bibliographystyle{alpha}
\bibliography{p1p1}

\end{document}